\documentclass[a4paper, reqno, 11pt]{amsart}
\usepackage{color}
\makeatletter

\@addtoreset{equation}{section}
\makeatother
\usepackage{setspace}
\usepackage{xcolor}
\usepackage{here}
\usepackage{graphicx}
\usepackage{float}
\usepackage{hyperref}
\usepackage[T1]{fontenc}
\usepackage{amsmath}
\usepackage{amssymb}
\usepackage{latexsym}
\usepackage{amsthm}
\usepackage{hyperref}
\usepackage{mathtools}
\usepackage{caption}
\everymath{\displaystyle}

\newtheorem{Thm}{Theorem}[section]

\newtheorem{Lem}[Thm]{Lemma}
\newtheorem{Prop}[Thm]{Proposition}

\theoremstyle{definition}

\newtheorem{Rem}[Thm]{Remark}

\begin{document}

\begin{abstract}
We study the maximum likelihood estimator of the location parameter of the Pearson Type VII distribution with known scale. 
We rigorously establish precise asymptotic properties such as strong consistency, asymptotic normality, Bahadur efficiency and asymptotic variance  of the maximum likelihood estimator. 
Our focus is  the heavy-tailed case, including the Cauchy distribution. 
The main difficulty lies in the fact that the likelihood equation may have multiple roots; nevertheless, the maximum likelihood estimator performs well for large samples.
\end{abstract}

\title[]{Asymptotics of the maximum likelihood estimator of the location parameter of Pearson Type VII distribution}
\author{Kazuki Okamura}
\date{\today}
\address{Department of Mathematics, Faculty of Science, Shizuoka University, 836, Ohya, Suruga-ku, Shizuoka, 422-8529, JAPAN.}
\email{okamura.kazuki@shizuoka.ac.jp}
\keywords{Pearson Type VII distribution, Cauchy distribution, maximum likelihood estimator, strong consistency, asymptotic normality, asymptotic efficiency, Bahadur efficiency}
\subjclass[2020]{Primary 62F12}
\maketitle

\section{Introduction}\label{sec:intro}

The family of Pearson Type VII distributions provides flexible heavy-tailed models.  
The estimation of its parameters dates back at least to Fisher~\cite{Fisher1922}, over a century ago, and many researchers have studied it since then; see Johnson, Kotz, and Balakrishnan~\cite[Section~28]{JKB1995} for a thorough survey of results prior to 1994. 
This class is also known as the location--scale family of Student's $t$ distributions or of $q$-Gaussian distributions. 
For estimating the location, the median is a robust alternative to the arithmetic mean; however it is not asymptotically efficient.

In general, the maximum likelihood estimator is widely regarded as optimal in large samples under standard regularity. 
Lange, Little, and Taylor \cite{LLT1989} proposed a strategy based on maximum likelihood for a general model with errors following the $t$-distribution and applied it to many problems. 
Under suitable regularity conditions, properties such as strong consistency, asymptotic normality, and Bahadur efficiency have been established by many researchers. 
For location--scale families, it is natural to consider the estimation of the location with known scale. 
The standard approach is to solve the likelihood equation explicitly or numerically, which often has a unique root. 
For the Cauchy distribution with known scale, however, the likelihood equation may have multiple roots (see Reeds~\cite{Reeds1985} for precise analysis), and the same phenomenon occurs for the Pearson Type VII distribution. 
For this reason, alternative estimators of the Cauchy location parameter have been considered. 
For example, Freue~\cite{Freue2007} considered the Pitman estimator for small samples, and Zhang~\cite{Zhang2014} considered an empirical Bayes estimator. 
Nevertheless, this does not represent a failure of the maximum likelihood estimator itself. 
Indeed, Bai and Fu~\cite{BaiFu1987} established its Bahadur efficiency.

In this paper, we deal with not only the Cauchy distribution but also the Pearson Type VII distribution and our focus is  the maximum likelihood estimator. 
Some references on the maximum likelihood estimator of the Pearson Type VII distribution are Borwein and Gabor \cite{BorweinGabor1984}, Tiku and Suresh \cite{TikuSuresh1992}, and Vaughan \cite{Vaughan1992}. 
We provide mathematically rigorous proofs of strong consistency, asymptotic efficiency, and Bahadur efficiency for the maximum likelihood estimator. 
Our approach does not analyze the likelihood equation directly. 
We show that the asymptotic properties of the maximum likelihood estimator mirror those for the arithmetic mean of independent and identically distributed (i.i.d.) random variables with finite variance. 
Asymptotically, the maximum likelihood estimator for the Pearson Type VII distribution performs well.

Now  we state the framework and the main result. 
Let $m > 1/2$, which covers the heavy-tailed regime of primary interest. 
Let $\textup{PVII}_{m}(\mu, \sigma)$ be the Pearson Type VII distribution with location $\mu \in \mathbb{R}$ and scale $\sigma > 0$. 
Then  the probability density function of $\textup{PVII}_{m}(\mu, \sigma)$ is given by 
\[ f(x) = c_{m} \frac{1}{\sigma} \left( 1 + \left(\frac{x-\mu}{\sigma}\right)^2 \right)^{-m},  \]
where $c_{m}$ is the normalizing constant, specifically, $\displaystyle c_{m} \coloneqq \left(\int_{\mathbb R} (1+x^2)^{-m} dx\right)^{-1}$. 
The case $m=1$ corresponds to the Cauchy distribution. 

We consider the maximum likelihood estimator of the location parameter of the Pearson Type VII distribution with known scale. 
We can assume that $\sigma = 1$. 
Let $(X_n)_{n \ge 1}$ be i.i.d. random variables on a complete probability space $(\Omega, \mathcal{F}, P)$ following $\textup{PVII}_{m}(\theta,1)$. 
Let $\hat{\theta}_n$ be the maximum likelihood estimator of the location parameter from a sample $(X_1, \dots, X_n)$ of size $n$. 
Let $\hat{\theta}_n (x_1, \dots, x_n)$ be a measurable function on $\mathbb{R}^n$ which maximizes the function $\displaystyle \theta \mapsto \prod_{i=1}^{n} f(x_i-\theta)$. 
Such a function exists by virtue of the measurable selection theorem. 
Let $\hat{\theta}_n \coloneqq \hat{\theta}_n (X_1, \dots, X_n)$. 

Our first main result is strong consistency. 

\begin{Thm}[Strong consistency]\label{thm:SLLN}
\[ \lim_{n \to \infty} \hat{\theta}_n = \theta,   \textup{ $P$-a.s.}   \]
\end{Thm}

We show this using the concept of the Fr\'echet mean. 

Once the strong consistency is given, it is natural to consider the asymptotic normality. 
We denote the normal distribution with mean $\mu$ and variance $\sigma^2$ by $N(\mu, \sigma^2)$. 

\begin{Thm}[Asymptotic normality]\label{thm:CLT}
$\left(\sqrt{n} (\hat{\theta}_n - \theta) \right)_n$ converges to \\
$N\left(0,\dfrac{m + 1}{m (2m -1)} \right)$ in distribution as $n \to \infty$. 
\end{Thm}

By Remark \ref{rem:FI} below, 
$I(\theta) = \dfrac{m (2m -1)}{m + 1}$, where $I(\theta)$ is the Fisher information for a single observation.

We proceed to the law of the iterated logarithm. 
It has connections with statistics, in particular with sequential testing. 
See \cite{Robbins1970,BoosSerfling1980,HeWang1995}. 

\begin{Thm}[Law of the iterated logarithm]\label{thm:LIL}
\[ \limsup_{n \to \infty} \sqrt{\frac{n}{\log \log n}} (\hat{\theta}_n - \theta)  = \sqrt{\frac{2(m + 1)}{m (2m -1)}},   \textup{ $P$-a.s.}   \]
\end{Thm}

For the proof, we use the technique of the deviation mean of i.i.d. random variables  investigated by Barczy and P\'ales \cite{BarczyPales2023} with some modifications.  

The following extends the result of Bai and Fu  \cite{BaiFu1987}, who considered the Cauchy distribution, to the Pearson Type VII distribution. 

\begin{Thm}[Bahadur efficiency and moderate deviation]\label{thm:BH} 
\!

\noindent(i) 
\begin{equation}\label{eq:BH-upper}
\limsup_{\epsilon \to +0} \frac{1}{\epsilon^2} \left( \limsup_{n \to \infty} \frac{\log P\left(\left|\hat{\theta}_n - \theta\right| > \epsilon\right)}{n} \right) \le -\frac{m (2m-1)}{2(m + 1)}.  
\end{equation}
\begin{equation}\label{eq:BH-lower}
\liminf_{\epsilon \to +0} \frac{1}{\epsilon^2} \left( \liminf_{n \to \infty} \frac{\log P\left(\left|\hat{\theta}_n - \theta\right| > \epsilon\right)}{n} \right) \ge -\frac{m (2m-1)}{2(m + 1)}.  
\end{equation}
(ii) For every sequence $(a_n)_n$ of positive numbers satisfying $\displaystyle \lim_{n \to \infty} a_n = \infty$ and $\displaystyle \lim_{n \to \infty} a_n / n^{1/2} = 0$ and every $\epsilon > 0$, 
\[ \lim_{n \to \infty} \frac{\log P\left(\left|\hat{\theta}_n - \theta\right| > \epsilon/a_n \right)}{n/a_n^2} = -\frac{m (2m-1)}{2(m + 1)}\epsilon^2.   \] 
\end{Thm}

This assertion implies Theorem \ref{thm:SLLN} and its proof does not depend on Theorem \ref{thm:SLLN}. 
However we can show Theorem \ref{thm:SLLN} much more easily than the proof of this assertion. 
For the proof, we follow the strategy of   \cite{BaiFu1987}.

It is worth investigating the probability that the estimator deviates significantly from the true value. 
In this paper, we let $\mathbb{N} \coloneqq \{1,2, \dots\}$. 

\begin{Thm}[Integrability]\label{thm:integrability}
There exist positive constants $r_{m}$ and  $N_{m} \in \mathbb{N}$ depending only on $m$ such that 
for every $r \ge r_{m}$ and every $n \ge N_{m}$, 
\[ P\left( \left|\hat{\theta}_n -\theta\right| > r \right) \le r^{-c_{m}^{\prime} n}, \]
where $c_m^{\prime} \coloneqq \lambda_{m}^{\prime} - \dfrac{(\lambda_{m}^{\prime})^2}{4} $ and $\lambda_m^{\prime} \coloneqq \min\left\{1, \dfrac{2m-1}{4} \right\}$. 
In particular, $\hat{\theta}_n \in L^{c_{m}^{\prime} n -1} (\Omega, \mathcal{F}, P)$ for $n \ge N_{m}$.  
\end{Thm}

We show this by modifying several estimates in the proof of Theorem \ref{thm:BH}. 

The Cram\'er-Rao inequality states that  for each $n \ge 1$, 
\[ n E\left[ \left(\hat{\theta}_n - \theta\right)^2\right] \ge \frac{1}{I(\theta)}.  \]
By this and Theorem \ref{thm:CLT}, 
it is natural to consider the large-sample asymptotics of $n E\left[ \left(\hat{\theta}_n - \theta\right)^2\right]$. 

\begin{Thm}[Variance asymptotics]\label{thm:AE}
\[ \lim_{n \to \infty} n E\left[ \left(\hat{\theta}_n - \theta\right)^2\right] = \frac{m + 1}{m (2m -1)}.  \]
\end{Thm}

This is consistent with  \cite[(28.61c)]{JKB1995}. 
We give a mathematically rigorous proof of it. 
The proof is technically involved and we use Theorem \ref{thm:CLT} and Theorem \ref{thm:integrability}, and an estimate obtained in the proof of Theorem \ref{thm:BH}. 

In the following sections, we present proofs of these assertions. 
In the final section, we give simulation studies. 

In the proofs of these results, we can assume that $\theta = 0$ without loss of generality. 
The parameter $m$ remains fixed throughout. 
Many constants will appear.
When a constant depends {\it only} on $m$, we indicate this by including $m$ as a subscript; otherwise we omit it even if it depends on $m$.

\section{Proof of Theorem \ref{thm:SLLN}}\label{sec:SLLN}

We first give an outline of the proof. 
We prove Theorem \ref{thm:SLLN} by following the strategy of Bhattacharya and Bhattacharya \cite[Section 3.2]{Bhattacharya2012}. 
One of our goals is to establish \cite[Theorem 3.3]{Bhattacharya2012} in the case where  the loss function is replaced\footnote{In \cite[Theorem 3.3]{Bhattacharya2012}, the loss function is given by the map $u \mapsto u^{\alpha}$ for $\alpha \in (0,1)$.} with the map $u \mapsto \log(1+u^2)$. 

The key step is Proposition \ref{prop:non-cpt-stable}, which shows that the Fr\'echet mean set $C_{Q_n (\omega)}$ (equivalently, the argmin set of $L_n$) is eventually contained in an arbitrarily small neighborhood of $0$. 
Proposition \ref{prop:non-cpt-stable} immediately yields Theorem \ref{thm:SLLN}. 
The proof of Proposition \ref{prop:non-cpt-stable} consists of four ingredients: 
uniform boundedness of minimizers (Lemma \ref{lem:bdd-minimizer-stable}), 
a uniform law of large numbers for $L_n$ on bounded intervals (Lemma  \ref{lem:as-cptunifconv-stable}), 
uniqueness of the population minimizer $C_{\nu_m} = \{0\}$ (Lemma \ref{lem:min-integral-stable}), 
and a stability lemma for minimizers of continuous functions with compact level sets (Lemma \ref{lem:min-noncpt-stable}).

Let 
\[ L_n (t) \coloneqq \frac{1}{n} \sum_{i=1}^{n} \log(1+ (X_i - t)^2), \ t \in \mathbb{R}.  \]
Let $\nu_m$ be the probability measure on the Lebesgue measurable space $(\mathbb{R}, \mathcal{L}(\mathbb{R}))$ of the Pearson Type VII distribution $\textup{PVII}_{m}(0,1)$, that is, 
\[ \nu_m (dx) \coloneqq c_{m} \left( 1 + x^2 \right)^{-m} dx. \]

\begin{Lem}\label{lem:metlower-am-stable}
There exists a positive constant $c_{m,1}$ depending only on $m$ such that $P$-a.s. $\omega$, 
there exists $N_1 (\omega) \in \mathbb{N}$ such that for every $n > N_1 (\omega)$ and every $t \in \mathbb{R}$ with $|t| \ge 2$, 
\[ L_n (t)(\omega) \ge \frac{c_{m,1}}{4} (\log(1+ t^{2}) - 2\log 2).  \]
\end{Lem}

\begin{proof}
Applying the inequality 
\begin{equation*} 
\log(1+x^{2}) + \log(1+y^{2}) \ge \frac{1}{2} \log(1+(x+y)^{2}), \ x, y \in \mathbb{R}, 
\end{equation*}
to $(x,y) = (X_i (\omega) - t, X_i (\omega))$, 
we see that 
\[ L_n (t)(\omega) \ge \frac{1}{n} \sum_{i=1}^{n}\log(1+(X_i (\omega) -t)^{2}) {\bf 1}_{[-1,1]}(X_i (\omega)) \]
\[ \ge \left(\frac{1}{2} \log(1+ t^2)- \log 2 \right) \frac{1}{n} \sum_{i=1}^{n} {\bf 1}_{[-1,1]}(X_i (\omega)). \]

By the strong law of large numbers,  
\[ \lim_{n \to \infty}  \frac{1}{n} \sum_{i=1}^{n} {\bf 1}_{[-1,1]}(X_i (\omega)) = \nu_{m}([-1,1]) > 0, \ \textup{ $P$-a.s. $\omega$}.  \]
We have the assertion for $c_{m,1} \coloneqq \nu_{m}([-1,1])$. 
\end{proof}

Denote the empirical distribution of $(X_i (\omega))_{i=1}^{n}$ by $Q_n (\omega)$, specifically, 
\[  Q_n (\omega) \coloneqq \frac{1}{n} \sum_{i=1}^{n} \delta_{X_i (w)}. \]

For a probability measure $\nu$ on the Lebesgue measurable space $(\mathbb{R}, \mathcal{L}(\mathbb{R}))$, 
let the expected loss function be 
\begin{equation*}
F_{\nu}(t) \coloneqq \int_{\mathbb R} \log (1+(x-t)^2) d\nu(x), \ t \in \mathbb{R}, 
\end{equation*}
and the mean set be 
\[ C_{\nu} \coloneqq \left\{t \in \mathbb{R} \middle| \min_{s \in \mathbb R} F_{\nu}(s) = F_{\nu}(t) \right\}. \]
Since $\log(1+(x-t)^2) = O\left(|x|^{\alpha} \right), x \to \infty$, for every $\alpha > 0$, 
$F_{\nu_m}(t) < \infty$ for every $t \in \mathbb{R}$. 

For $\nu = Q_n (\omega)$, $C_{Q_n (\omega)}$ is called the Fr\'echet mean set. 
Recall that maximizing the likelihood is equivalent to minimizing the empirical negative log-likelihood $L_n$. 
For the empirical measure $Q_n$, the corresponding Fr\'echet function $F_{Q_n (\omega)} (t)$ equals $L_n (t)(\omega)$. 
Therefore, 
$\hat{\theta}_n (\omega) \in C_{Q_n (\omega)}$. 

Another goal of this section is to show that the mean set $C_{\nu_m}$ is a singleton, which will be established in Lemma \ref{lem:min-integral-stable} below. 
Let 
\begin{equation}\label{eq:def-F-alpha} 
F_{m} (t) \coloneqq F_{\nu_{m}} (t) = \int_{\mathbb R} \log\left(1+(x-t)^2\right) \nu_{m}(dx).
\end{equation} 
This is the expected loss function. 

\begin{Lem}[boundedness of minimizers]\label{lem:bdd-minimizer-stable}
There exists a positive constant $r_{m,1}$ depending only on $m$ such that $P$-a.s. $\omega$, 
there exists $N_2 (\omega) \in \mathbb{N}$ such that for every $n > N_2 (\omega)$, 
$C_{Q_n (\omega)} \subset [-r_{m,1}, r_{m,1}]$.  
\end{Lem}

\begin{proof}
By Lemma \ref{lem:metlower-am-stable}, 
there exists an event $\Omega_1$  such that $P(\Omega_1) = 1$ and for every $\omega \in \Omega_1$, 
there exists $N_1 (\omega) \in \mathbb{N}$ such that for every $n > N_1 (\omega)$ and every $t \in \mathbb{R}$ with $|t| \ge 2$, 
\[ L_n (t) (\omega) \ge \frac{c_{m,1}}{4} (\log(1+ t^{2}) - 2\log 2).  \]

Assume that $\omega \in \Omega_1$, $t_n \in C_{Q_n (\omega)}$ and $|t_n| \ge 2$.  
Then, for every $n > N_1 (\omega)$, 
\[ L_n (0)(\omega) \ge  L_n (t_n) (\omega) \ge \frac{c_{m,1}}{4} (\log(1+ t_n^{2}) - 2\log 2)  \]

Recall \eqref{eq:def-F-alpha}. 
We remark that $F_{\nu_m}(0) < \infty$. 
By the strong law of large numbers, 
there exists an event $\Omega_2 \subset \Omega_1$ such that $P(\Omega_2) = 1$ and for every $\omega \in \Omega_2$, 
\[ \lim_{n \to \infty} L_n (0)(\omega) = F_{m}(0) < +\infty.  \]
In particular, 
there exists $N_2 (\omega) > N_1 (\omega)$ such that for every $n > N_2 (\omega)$, 
\[ L_n (0)(\omega)   \le 1 +  F_{m}(0).  \]
Hence, there exists a constant $r_{m,1} > 1$ such that for every $\omega \in \Omega_2$ and $n > N_2 (\omega)$, 
$|t_n| < r_{m,1}$. 
\end{proof}

\begin{Lem}[a.s. pointwise convergence]\label{lem:as-pointwise-stable}
$P$-a.s., it holds that for every $t \in \mathbb{R}$, 
\begin{equation}\label{eq:as-pointwise-stable} 
\lim_{n \to \infty} L_n (t) = F_{m}(t). 
\end{equation}
\end{Lem}

\begin{proof} 
By the strong law of large numbers, for every \textit{fixed} $t \in \mathbb R$, equation \eqref{eq:as-pointwise-stable} holds a.s. 
More specifically, for every $t \in \mathbb R$, 
there exists an event $\Omega_t$ such that $P(\Omega_t) = 1$ and for every $\omega \in \Omega_t$, 
$\lim_{n \to \infty} L_n (t)(\omega) = F_{m}(t)$. 

We use the Lipschitz continuity\footnote{This estimate works well if $x$ and $y$ are close. If $x$ or $y$ is large, the bound can be very loose.} of $\log(1+x^2)$, specifically, 
\begin{equation}\label{eq:log-Lip-basic}
\left|\log(1+x^2) - \log(1+y^2)\right| \le \left||x|-|y|\right| \le |x-y|. 
\end{equation}
Hence
\begin{equation}\label{eq:Lip-Ln} 
\left|L_n (t)(\omega) - L_n (s)(\omega) \right| \le |t-s|, \ \ t, s \in \mathbb{R}, \, n \ge 1, 
\end{equation}
and 
\begin{equation}\label{eq:Lip-Fm}  
\left|F_m (t) - F_m (s) \right| \le |t-s|, \ t, s \in \mathbb{R}. 
\end{equation} 

We use the rational approximation. 
Let $\displaystyle \Omega_{\mathbb Q} \coloneqq \bigcap_{t \in \mathbb{Q}} \Omega_t$. 
Then $P(\Omega_{\mathbb Q}) = 1$. 
Take $t \in \mathbb{R}$ arbitrarily. 
By \eqref{eq:Lip-Ln} and \eqref{eq:Lip-Fm}, 
it holds that for every $\omega \in \Omega_{\mathbb Q}$ and $s \in \mathbb{Q}$, 
\[ \limsup_{n \to \infty} |L_n (t)(\omega) - F_{m}(t)| \le 2|t-s| +  \limsup_{n \to \infty} |L_n (s)(\omega) - F_{m}(s)| = 2|t-s|. \]
Since $s \in \mathbb{Q}$ can be taken arbitrarily close to $t$, 
we see that $ \lim_{n \to \infty} |L_n (t)(\omega) - F_{m}(t)| = 0$. 
\end{proof}

The following is the uniform law of large numbers. 

\begin{Lem}\label{lem:as-cptunifconv-stable}
$P$-a.s., it holds that for every compact subset $K$ of $\mathbb{R}$, 
\[ \lim_{n \to \infty} \max_{t \in K} \left|L_n (t) -  F_{m} (t) \right| = 0. \]
\end{Lem}

\begin{proof}
By Lemma \ref{lem:as-pointwise-stable}, 
there exists an event $\Omega_3$  such that $P(\Omega_3) = 1$ and for every $\omega \in \Omega_3$, 
 \eqref{eq:as-pointwise-stable} holds for every $t \in \mathbb{R}$. 

Let $\omega \in \Omega_3$. 
Let $\epsilon > 0$ arbitrarily. 
Then, by \eqref{eq:log-Lip-basic}, 
for each $t_1, t_2 \in \mathbb R$ with $|t_1 - t_2| < \epsilon/4$,  
\[ \left| F_{m} (t_1)  - F_{m} (t_2)   \right| \le |t_1 - t_2| \le \frac{\epsilon}{4}. \]

Let $u_1, \cdots, u_{\ell}$ be points in $K$ such that $K \subset \cup_{j=1}^{\ell} (u_j - \epsilon/4, u_j + \epsilon/4)$. 
Then, by Lemma \ref{lem:as-pointwise-stable}, 
there exists $N_3 (\omega) \in \mathbb{N}$ such that for every $n > N_3 (\omega)$, 
\[ \max_{1 \le j \le \ell} \left|L_n (u_j)(\omega) - F_{m} (u_j)   \right| < \frac{\epsilon}{4}. \]

Then, by \eqref{eq:log-Lip-basic},  if $t \in K$ and $|t - u_j| < \epsilon/4$, then, for every $n > N_3 (\omega)$, 
\[ \left|L_n (t) (\omega) - F_{m} (t)  \right|  \le \frac{\epsilon}{2} + \left| L_n (u_j)(\omega) - L_n (t) (\omega)  \right| \le \epsilon.  \]
\end{proof}

\begin{Lem}\label{lem:min-integral-stable}
The function $F_{m}$ is  strictly decreasing on $(-\infty,0)$ and  strictly increasing on $(0,\infty)$. 
In particular, 
$C_{\nu_{m}} = \{0\}$. 
\end{Lem}

\begin{proof}
By the Lebesgue convergence theorem, we see that 
\begin{equation}\label{eq:Fm-derivative} 
F_{m}^{\prime}(t) = -2 c_{m} \int_{\mathbb R} \frac{x-t}{(1+(x-t)^2)(1+x^2)^{m}} dx. 
\end{equation} 
By change of variables, 
\[ \int_{\mathbb R} \frac{x-t}{(1+(x-t)^2)(1+x^2)^{m}} dx = \int_{\mathbb R} \frac{x}{(1+x^2)(1+ (x+t)^2)^{m}} dx \]
\[ = \int_{0}^{\infty} \frac{x}{(1+x^2)(1+ (x+t)^2)^{m}} dx + \int_{-\infty}^{0} \frac{x}{(1+x^2)(1+ (x+t)^2)^{m}} dx\]
\[ = \int_{0}^{\infty} \frac{x}{1+x^2} \left(\frac{1}{(1+ (x+t)^2)^{m}} - \frac{1}{(1+ (x-t)^2)^{m}} \right) dx. \]
The last integral is positive if $t < 0$, is zero if $t = 0$, and is negative if $t > 0$. 
Hence, the sign of $F_{m}^{\prime}(t)$ is equal to the sign of $t$, and hence, $F_{m}(t)$ takes its minimum only at $t = 0$. 
\end{proof}

For a non-empty subset $A$ of $\mathbb{R}$, let 
\[ d(x,A) \coloneqq \inf\left\{|x-y| : y \in A \right\}. \]

\begin{Lem}\label{lem:min-noncpt-stable}
Let $\varphi$ be a continuous function on $\mathbb R$ such that 
$\displaystyle \lim_{|z| \to \infty} \varphi(z) = \infty$.  
Let $\displaystyle C_{\varphi} \coloneqq \left\{x \in \mathbb{R} : \min_{t \in \mathbb{R}} \varphi(t) = \varphi(x) \right\}$. 
Then, 
for every $\epsilon > 0$, 
there exists $\delta > 0$ such that 
for every $x \in \mathbb{R}$ with $\displaystyle \varphi(x) \le \min_{t \in \mathbb{R}} \varphi(t) + \delta$, 
$d(x,C_{\varphi}) < \epsilon$. 
\end{Lem}

\begin{proof}
We show this by contradiction. 
Assume that there exists $\epsilon_0 > 0$ such that for every $n \in \mathbb{N}$, 
there exists $x_n \in \mathbb{R}$ such that $\varphi(x_n) \le \min_{t \in \mathbb{R}} \varphi(t) + 1/n$ and $d(x_n, C_{\varphi}) \ge \epsilon_0$.  
Since $\displaystyle \sup_{n \in \mathbb{N}} \varphi(x_n) < +\infty$, 
by the assumption of $\varphi$, 
$(x_n)_n$ is a bounded sequence. 
Then  there exist a subsequence $(x_{n_k})_k$ and a point $z \in \mathbb{R}$ such that $x_{n_k} \to z, k \to \infty$. 
By the continuity of $\varphi$, 
$\displaystyle \varphi(z) = \lim_{k \to \infty} \varphi(x_{n_k}) = \min_{t \in \mathbb{R}} \varphi(t)$. 
Hence, $z \in C_{\varphi}$. 
Now it suffices to recall that $|x_{n_k} - z| \ge d(x_{n_k}, C_{\varphi}) \ge \epsilon_0$ for each $k$.  
\end{proof}

\begin{Prop}[confinement of minimizers]\label{prop:non-cpt-stable}
$P$-a.s. $\omega$, it holds that for every $\epsilon > 0$, 
there exists $N_4 (\omega, \epsilon) \in \mathbb{N}$ such that for every $n > N_4 (\omega, \epsilon)$, 
$C_{Q_n (\omega)} \subset [-\epsilon, \epsilon]$. 
\end{Prop}

\begin{proof}
By applying Lemma \ref{lem:bdd-minimizer-stable} and Lemma \ref{lem:as-cptunifconv-stable} to $K = [-r_{m,1}, r_{m,1}]$, 
it holds that for every $\omega \in \Omega_3$ and $t_n \in C_{Q_n  (\omega)}$, 
\[ \lim_{n \to \infty} \left|L_n (t_n)(\omega) -F_{m} (t_n) \right| = 0.  \]
Let $\epsilon > 0$. 
Then  
there exists $N_4 (\omega, \epsilon) \in \mathbb{N}$ such that for every $n > N_4 (\omega, \epsilon)$ and $t_n \in C_{Q_n  (\omega)}$, 
\[ F_{m}(t_n) \le L_n (t_n) (\omega)  + \frac{\epsilon}{4} \]
and 
\[ F_{m}(0)  \ge L_n (0) (\omega) - \frac{\epsilon}{4}, \]
in particular, 
\[ L_n (t_n) \le F_{m} (0) + \epsilon.\] 
Now apply Lemma \ref{lem:min-noncpt-stable} to $\displaystyle \varphi = F_{m}$ and Lemma \ref{lem:min-integral-stable}. 
\end{proof}

Recall that $\hat{\theta}_n$ is a measurable selection from $C_{Q_n (\omega)}$. 
By Proposition \ref{prop:non-cpt-stable},  
we obtain Theorem \ref{thm:SLLN}. 

\begin{Rem}
Recently, Sch\"otz \cite{Schotz2022} gave a precise analysis for the Fr\'echet mean. 
His approach uses a general ergodic theorem and differs from the above approach. 
\end{Rem}

\section{Proof of Theorem \ref{thm:CLT}}\label{sec:CLT}

We first give an outline of the proof. 
We follow the strategy of Barczy and P\'ales \cite[Section 4]{BarczyPales2023}. 
Let 
\[ D(x,t) \coloneqq \frac{x-t}{1+(x-t)^2}, \ x, t \in \mathbb{R},  \]
and 
\[ D_n (t) \coloneqq \frac{1}{n} \sum_{i=1}^{n}  D(X_i, t), \ t \in \mathbb{R}. \]
Then $-2D_{n}(t) \equiv L_{n}^{\prime}(t)$ and hence, the likelihood equation is $D_n (t) = 0$. 

Since the map $t \mapsto D(x,t)$ is not monotone on $\mathbb{R}$ for each fixed $x$, 
we cannot apply the result of \cite[Section 4]{BarczyPales2023} directly. 
Therefore we ``localize'' the argument. 
Specifically, we construct a sequence of events  $\mathcal{A}_n, n \ge 1$, defined in \eqref{eq:def-good-n} below, such that $\displaystyle \lim_{n \to \infty} P(\mathcal{A}_n) = 1$, and, on each $\mathcal{A}_n$ the map $t \mapsto D_n (t)$ is strictly decreasing on a fixed interval $I$ and has a unique zero in $I$. 
This yields that on each $\mathcal{A}_n$, $\hat{\theta}_n > t$ if and only if $D_n (t) > 0$ for $t \in I$. 
Once this localization is established, the remainder of the proof is standard. 
Consider a Taylor expansion of $D_n (y/\sqrt{n})$ at $0$ and apply the central limit theorem, the law of large numbers and Slutsky's lemma. 

The arguments in Section \ref{sec:SLLN} are not sufficient for this localization, 
since the possibility that $|C_{Q_n} \cap [-T, T]| \ge 2$ has not yet been excluded. 
We overcome this issue by Proposition \ref{prop:non-cpt-unique-stable}. 
For the proof, 
we compare $D_n(t)$ and its derivative $D_n^{\prime}(t)$ with their population counterparts $G_m(t)$ and $G_m^{\prime} (t)$ defined below. 
Lemma \ref{lem:lv2-Azuma} shows that $D_n (t)$ is uniformly close to $G_m (t)$ for large $n$. 
Lemma \ref{lem:decreasing-origin} shows that $G_m^{\prime} (t) < 0$. 
Lemma  \ref{lem:lv3-Azuma} transfers this to $D_n^{\prime}(t)$ for large $n$. 

Theorem \ref{thm:CLT} can also be shown by using van der Vaart \cite[Theorem 5.23]{Vaart1998}. 
See Remark \ref{rem:FI} (iii) below for details. 
However, throughout this paper we repeatedly use the notation such as  $\{\mathcal{A}_n\}_{n \ge 1}$ and refer to related assertions in this section, we therefore present the full details here.  

Let 
\[ G_{m}(t) \coloneqq E\left[ D(X_1, t) \right] = \int_{\mathbb R} D(x,t) \nu_{m}(dx). \]
Then, by \eqref{eq:Fm-derivative}, 
$-2G_{m}(t) \equiv F_{m}^{\prime}(t)$ and hence, $G_{m}(-t) > 0 > G_{m}(t)$ for every $t > 0$.

We show the following: 
\begin{Lem}\label{lem:lv2-Azuma}
For every $\epsilon > 0$, there exists a positive constant $c_{\epsilon,1}$ depending on $\epsilon$ such that for every $n \ge 1$,  
\[ P\left(  \max_{t \in [-1,1]} \left| D_n (t)  - G_{m}(t) \right| > \epsilon \right)  \le c_{\epsilon,1} \exp\left(- \frac{\epsilon^2}{2} n \right). \]
In particular, 
\[ \lim_{n \to \infty} \max_{t \in [-1,1]} \left| D_n (t) - G_{m}(t) \right| = 0, \ \textup{ $P$-a.s.}  \]
\end{Lem}

The constant $c_{\epsilon,1}$ is independent of $m$. 

\begin{proof}
Let $t \in [-1,1]$. 
Let $Y_i \coloneqq D(X_i, t) - G_{m}(t)$. 
Since $|D(X_i, t)| \le 1/2$ and $|G_m (t)| \le 1/2$, 
it holds that $(Y_i)_i$ are i.i.d., $|Y_i| \le 1$ and $E[Y_i] = 0$. 
By the Azuma-Hoeffding inequality (see Petrov \cite[2.6.2]{Petrov1995} or Boucheron, Lugosi and Massart \cite[Theorem 2.8]{BLM2013}), 
\[ P\left(  \left| D_n (t) - G_{m}(t) \right| > \epsilon \right) = P\left(  \left| \sum_{i=1}^{n} Y_i \right| > n\epsilon \right) \le 2\exp\left(- \frac{\epsilon^2}{2} n \right). \]

Let $\mathcal{D}_N \coloneqq \{\ell/N : -N \le \ell \le N\}$ for $N \in \mathbb{N}$. 
Since $t \mapsto D(x,t)$ is Lipschitz continuous with the Lipschitz constant $1$, 
\[  \max_{t \in [-1,1]} \left| D_n (t) - G_{m}(t) \right| \le  \max_{t \in \mathcal{D}_N} \left| D_n (t) - G_{m}(t) \right|  + \frac{2}{N}. \]
Hence, for $N > 4/\epsilon$ and $n \ge 1$, 
\[ P\left(  \max_{t \in [-1,1]} \left| D_n (t)  - G_{m}(t) \right| > \epsilon \right)  \le P\left(  \max_{t \in \mathcal{D}_N} \left| D_n (t) - G_{m}(t) \right| > \frac{\epsilon}{2} \right)\]
\[ \le \sum_{t \in \mathcal{D}_N} P\left(   \left| D_n (t) - G_{m}(t) \right| > \frac{\epsilon}{2} \right) \le 2(2N+1) \exp\left(- \frac{\epsilon^2}{2} n \right).\]
Now use the Borel-Cantelli lemma and then let $\epsilon \to +0$, and we obtain the a.s. convergence. 
\end{proof}

We see that 
\[ \partial_t D(x,t) = \frac{(x-t)^2 -1}{(1+ (x-t)^2)^2}. \]

\begin{Lem}\label{lem:decreasing-origin}
There exists a constant $r_{m,2} \in (0,1)$ such that $G_{m}^{\prime}(t) < 0$ for every $t \in [-r_{m,2}, r_{m,2}]$. 
\end{Lem}

\begin{proof}
By the Lebesgue convergence theorem, 
we see that 
\[ G_{m}^{\prime}(t) = \int_{\mathbb R} \partial_t D(x,t) \nu_{m}(dx). \]
By the Lebesgue convergence theorem again, 
we see that $G_{m}^{\prime}$ is continuous. 
Hence, it suffices to show that $G_{m}^{\prime}(0) < 0$. 

By the change of variables $x = \tan\theta$, 
\[ \int_{\mathbb R} \frac{x^2 -1}{(1+x^2)^{2+m}} dx = -\int_{-\pi/2}^{\pi/2} \cos^{2m} \theta \cos(2\theta) d\theta =  -2\int_{0}^{\pi/2} \cos^{2m} \theta \cos(2\theta) d\theta.   \]
We see that 
\[ \int_{0}^{\pi/2} \cos^{2m} \theta \cos(2\theta) d\theta = \int_{0}^{\pi/4} \cos^{2m} \theta \cos(2\theta) d\theta + \int_{\pi/4}^{\pi/2} \cos^{2m} \theta \cos(2\theta) d\theta\]
\[  = \int_{0}^{\pi/4} \cos^{2m} \theta \cos(2\theta) d\theta - \int_{0}^{\pi/4} \cos^{2m} \left(\frac{\pi}{2} -\theta \right) \cos(2\theta) d\theta > 0. \]
\end{proof}

We also deal with the derivatives of $D_n (t)$ and $G_m (t)$ with respect to $t$. 
The following corresponds to \cite[(3.32)]{BaiFu1987}. 
\begin{Lem}\label{lem:lv3-Azuma}
For every $\epsilon > 0$, there exists a positive constant $c_{\epsilon,2}$  depending on $\epsilon$  such that for every $n \ge 1$,  
\[ P\left(  \max_{t \in [-1,1]} \left| D_n^{\prime} (t)  - G_{m}^{\prime} (t) \right| > \epsilon \right)  \le c_{\epsilon,2} \exp\left(- \frac{\epsilon^2}{12} n \right). \]
In particular, 
\[ \lim_{n \to \infty} \max_{t \in [-r_{m,2}, r_{m,2}]} \left| D_n^{\prime} (t) - G_{m}^{\prime}(t) \right| = 0, \ \textup{ $P$-a.s.}  \]
\end{Lem}

As in Lemma \ref{lem:lv3-Azuma}, the constant $c_{\epsilon,2}$ is also independent of $m$. 

\begin{proof}
By 
\begin{equation}\label{eq:2nd-derivative-D} 
\partial_t^2 D(x,t) = \frac{2(x-t)((x-t)^2 - 3)}{(1+(x-t)^2)^3}, 
\end{equation}
$|\partial_t^2 D(x,t)| \le 3$, and hence, 
the map $t \mapsto \partial_t D(x,t)$ is Lipschitz continuous with the Lipschitz constant $3$. 
Let $Y^{\prime}_i \coloneqq \partial_t D(X_i,t) - G_{m}^{\prime}(t)$. 
Since $|\partial_t  D(X_i, t)| \le 1$ and $|G_m^{\prime} (t)| \le 1$, 
$\left(Y^{\prime}_i \right)_i$ are i.i.d., $|Y_i^{\prime}| \le 2$ and $E[Y_i^{\prime}] = 0$. 
Therefore, we can show this assertion as in the proof of Lemma \ref{lem:lv2-Azuma}. 
\end{proof}

We remark that $C_{Q_n (\omega)} \ne \emptyset$ and 
\[ C_{Q_n (\omega)} \subset \left\{t \in \mathbb{R} \middle| D_n (t) (\omega) = 0  \right\}. \]

\begin{Prop}\label{prop:non-cpt-unique-stable}
$P$-a.s. $\omega$, there exists $N_5 (\omega) \in \mathbb{N}$ such that for every $n > N_5 (\omega)$, 
$\left|C_{Q_n (\omega)} \cap [-r_{m,2}, r_{m,2}] \right| = 1$.  
\end{Prop}

\begin{proof}
By Proposition \ref{prop:non-cpt-stable}, 
it holds that $P$-a.s. $\omega$, 
for $n \ge N_4 (\omega, r_{m,2})$, 
$|C_{Q_n (\omega)} \cap [-r_{m,2}, r_{m,2}]| = |C_{Q_n (\omega)}| \ge 1$.

Let 
\[ c_{m,2} \coloneqq \frac{1}{2} \min_{t \in [-r_{m,2}, r_{m,2}]} -G_{m}^{\prime}(t), \]
which is positive by Lemma \ref{lem:decreasing-origin}. 

By Lemma \ref{lem:lv3-Azuma}, 
it holds that $P$-a.s. $\omega$, 
there exists $N_{6}(\omega) \in \mathbb{N}$ such that for every $n > N_{6}(\omega)$, 
\[ \max_{t \in [-r_{m,2}, r_{m,2}]} D_n^{\prime} (t) (\omega) \le -c_{m,2}, \]  
in particular, $D_n (t) (\omega) $ is strictly decreasing in $t$ on $[-r_{m,2}, r_{m,2}]$.

Furthermore, by Lemma \ref{lem:lv2-Azuma}, 
it holds that $P$-a.s. $\omega$, 
there exists $N_{7}(\omega) \in \mathbb{N}$ such that for every $n > N_{7}(\omega)$, 
\[ D_n (-r_{m,2}) (\omega) > 0 > D_n (r_{m,2})(\omega).  \]  

By the intermediate value theorem, 
it holds that $P$-a.s. $\omega$, 
there exists $N_{8}(\omega) \in \mathbb{N}$ such that for every $n > N_{8}(\omega)$, 
\[ \left| \left\{t \in [-r_{m,2}, r_{m,2}]  \middle| D_n (t) (\omega) = 0  \right\} \right| = 1,  \]
which implies 
$|C_{Q_n (\omega)} \cap [-r_{m,2}, r_{m,2}]| \le 1$. 
\end{proof}

Let $\mathcal{A}_{n,1}$ be the event that $D_n (-r_{m,2}) > 0 > D_n (r_{m,2})$. 
Let $\mathcal{A}_{n,2}$ be the event that $D_n^{\prime}(t) \le -c_{m,2}/2$ for every $t \in [-r_{m,2}, r_{m,2}]$. 
Let $\mathcal{A}_{n,3}$ be the event that $|C_{Q_n}| = 1$ and $\hat{\theta}_n \in [-r_{m,2}/2, r_{m,2}/2]$. 
Let 
\begin{equation}\label{eq:def-good-n}
\mathcal{A}_{n} \coloneqq \mathcal{A}_{n,1} \cap \mathcal{A}_{n,2} \cap \mathcal{A}_{n,3}. 
\end{equation} 
Let 
\[ \widetilde{\mathcal{A}_i} \coloneqq \bigcup_{N \ge 1} \bigcap_{n \ge N} \mathcal{A}_{n,i}, \ \ i = 1,2,3.\] 

By Lemma \ref{lem:lv2-Azuma}, $\displaystyle P\left( \widetilde{\mathcal{A}_1} \right) = 1$. 
By Lemma \ref{lem:lv3-Azuma}, $\displaystyle P\left( \widetilde{\mathcal{A}_2} \right) = 1$. 
By Propositions \ref{prop:non-cpt-stable} and \ref{prop:non-cpt-unique-stable}, $\displaystyle P\left( \widetilde{\mathcal{A}_3} \right) = 1$. 
Since $\displaystyle \widetilde{\mathcal{A}_1} \cap \widetilde{\mathcal{A}_2} \cap \widetilde{\mathcal{A}_3} =  \bigcup_{N \ge 1} \bigcap_{n \ge N} \mathcal{A}_{n}$, 
$\displaystyle P\left( \bigcup_{N \ge 1} \bigcap_{n \ge N} \mathcal{A}_{n} \right) = 1$, 
and in particular, 
$\displaystyle \lim_{n \to \infty} P(\mathcal{A}_n) = 1$. 

For every $t \in (-r_{m,2} /2, r_{m,2} /2)$, on $\mathcal{A}_n$, 
$\hat{\theta}_n < t$ if and only if $D_n (t) < 0$.

Let $y \in \mathbb{R}$.  
Then  
\[ \lim_{n \to \infty} P\left(\sqrt{n}\hat{\theta}_n < y\right) - P\left(\left\{\sqrt{n}\hat{\theta}_n < y \right\} \cap \mathcal{A}_n\right) = 0, \]
and 
\[ \lim_{n \to \infty} P\left( D_n \left(\frac{y}{\sqrt{n}}\right) < 0\right) - P\left(\left\{ D_n \left(\frac{y}{\sqrt{n}}\right) < 0 \right\} \cap \mathcal{A}_n\right) = 0. \]
Since 
\[  P\left(\left\{\sqrt{n}\hat{\theta}_n < y \right\} \cap \mathcal{A}_n\right) = P\left(\left\{ D_n \left(\frac{y}{\sqrt{n}}\right) < 0 \right\} \cap \mathcal{A}_n\right) \]
for every $n$ satisfying that $n > 4y^2$, 
\[ \lim_{n \to \infty} P\left(\sqrt{n}\hat{\theta}_n < y\right) - P\left( D_n \left(\frac{y}{\sqrt{n}}\right) < 0\right) = 0. \]

Hence, it suffices to show that 
\begin{equation}\label{eq:dm-limit-1}
 \lim_{n \to \infty} P\left(D_n \left(\frac{y}{\sqrt{n}}\right) < 0\right) = \int_{-\infty}^y \varphi_{m} (t) dt, 
\end{equation}
where $\varphi_{m}$ is the density function of the distribution $\displaystyle N\left(0,\frac{m + 1}{m (2m -1)} \right)$. 

It holds that 
\[ \sqrt{n} D_n \left(\frac{y}{\sqrt{n}}\right) = \sqrt{n} D_n \left(0 \right) + \frac{y}{n} \sum_{i=1}^{n} \partial_t D (X_i, 0) \]
\[ + \frac{1}{\sqrt{n}} \sum_{i=1}^{n} \left(D\left(X_i, \frac{y}{\sqrt{n}}\right) - D\left(X_i, 0\right)  - \frac{y}{\sqrt{n}} \partial_t D (X_i, 0)\right).  \]

By symmetry, 
\begin{equation}\label{eq:mean-exp-D-zero} 
E\left[D(X_1, 0) \right] = c_{m} \int_{\mathbb R} \frac{x}{(1+x^2)^{1+m}} dx = 0. 
\end{equation} 

By the change of variables $x = \tan \theta$, 
\begin{equation}\label{eq:squared-exp-D} 
E\left[D(X_1, 0)^2 \right] = c_{m} \int_{\mathbb R} \frac{x^2}{(1+x^2)^{2+m}} dx = \frac{B(3/2, m + 1/2)}{B(1/2, m - 1/2)} =  \frac{2m -1}{4m (m +1)}, 
\end{equation} 
where $B(\cdot, \cdot)$ is the beta function. 
Hence, 
\begin{equation}\label{eq:CLT-core} 
\sqrt{n} D_n \left(0 \right)  \Rightarrow N\left(0, \frac{2m -1}{4m (m +1)} \right), \ n \to \infty, 
\end{equation}
where $\Rightarrow$ denotes the convergence in distribution. 

It holds that 
\[ E\left[  |\partial_t D(X_1,0)| \right] \le c_{m} \int_{\mathbb R} \frac{1}{(1+x^2)^{1+m}} dx < \infty, \]
and 
\begin{align}\label{eq:mean-partial-D} 
E\left[  \partial_t D(X_1,0) \right] &= c_{m} \int_{\mathbb R} \frac{x^2 -1}{(1+x^2)^{2+m}} dx \notag \\ 
&= \frac{B(3/2, m + 1/2)}{B(1/2, m - 1/2)} - \frac{B(1/2, m + 3/2)}{B(1/2, m - 1/2)} = -\frac{2m -1}{2(m + 1)}. 
\end{align}

Hence, by the strong law of large numbers, 
\begin{equation}\label{eq:CLT-const} 
\lim_{n \to \infty} \frac{y}{n} \sum_{i=1}^{n} \partial_t D (X_i, 0) = -\frac{2m -1}{2(m + 1)}y, \ \textup{ $P$-a.s.}  
\end{equation}

By \eqref{eq:2nd-derivative-D}, 
\[ \max_{x,t \in \mathbb{R}} \left|  \partial_t^2 D(x,t) \right| = \max_{y \in \mathbb{R}} \frac{2|y| |y^2 - 3|}{(1+y^2)^3} \eqqcolon C_1 < \infty.  \]
By this and the mean value theorem, 
it holds\footnote{This holds without any exceptional set.} that 
\begin{equation}\label{eq:non-as-estimate} 
\left|  D\left(X_i, \frac{y}{\sqrt{n}}\right) - D\left(X_i, 0\right)  - \frac{y}{\sqrt{n}} \partial_t D (X_i, 0) \right| \le C_1 \frac{|y|^2}{n}. 
\end{equation}
Hence, 
\begin{equation}\label{eq:CLT-vanish} 
\lim_{n \to \infty} \frac{1}{\sqrt{n}} \sum_{i=1}^{n} \left(D\left(X_i, \frac{y}{\sqrt{n}}\right) - D\left(X_i, 0\right)  - \frac{y}{\sqrt{n}} \partial_t D (X_i, 0) \right) = 0, \  \textup{ $P$-a.s.}  
\end{equation}

By \eqref{eq:CLT-core}, \eqref{eq:CLT-const}, \eqref{eq:CLT-vanish} and Slutsky's lemma, 
\[  \sqrt{n} D_n \left(\frac{y}{\sqrt{n}}\right) \Rightarrow N\left(-\frac{2m -1}{2(m + 1)}y,   \frac{2m -1}{4m (m +1)} \right), \ n \to \infty.  \]
Thus we see that \eqref{eq:dm-limit-1} holds and the proof of Theorem \ref{thm:CLT} is completed. 

Subsection \ref{subsec:AN-simulation} below provides numerical verifications of Theorem \ref{thm:CLT} by using the Kolmogorov--Smirnov distance. 
Subsection \ref{subsec:CI} below provides confidence intervals of the parameter $\theta$ by using $\hat{\theta}_n$. 

\begin{Rem}\label{rem:FI}
(i) By \eqref{eq:squared-exp-D}, 
the Fisher information is given by 
\[ I(0) = E\left[ \left( \frac{\partial}{\partial t} \log f(X_1, t) \Bigg|_{t=0} \right)^2 \right] = 4m^2 E\left[ D(X_1, 0)^2 \right] = \frac{m(2m -1)}{m +1}. \]
(ii) The likelihood equation $D_n (t) = 0$ does not depend on the parameter $m$. 
For $m=1$, \cite{Reeds1985} shows that for each $k \ge 0$, 
\[ \lim_{n \to \infty} P\left( \left| \{t \in \mathbb{R} | D_n (t) = 0\}\right| = 2k+1 \right) =  \exp\left(-\frac{1}{\pi}\right) \frac{1}{k! \pi^k}. \] 
Here $c_1 = 1/\pi$, and we conjecture that for each $m > 1/2$ and $k \ge 0$, 
\[ \lim_{n \to \infty} P\left( \left| \{t \in \mathbb{R} | D_n (t) = 0\}\right| = 2k+1 \right) = \exp\left(-c_m\right) \frac{c_m^k}{k!}. \] 
(iii) We can apply \cite[Theorem 5.23]{Vaart1998}. 
We confirm the assumptions of the assertion. 
Let $\mathfrak{m}(x, \theta) \coloneqq -\log(1+(x-\theta)^2), \ x, \theta \in \mathbb{R}$. 
Then the MLE $\hat{\theta}_n$ maximizes the map $\theta \mapsto \sum_{i=1}^{n} \mathfrak{m}(X_i, \theta)$. 
We see that $\mathfrak{m} \in C^{\infty}(\mathbb{R}^2)$ and  by \eqref{eq:log-Lip-basic}, 
$|\mathfrak{m}(x, \theta_1) - \mathfrak{m}(x, \theta_2)| \le |\theta_1 - \theta_2|$. 
It holds that $E\left[ \mathfrak{m}(X_1, \theta) \right] = -F_m (\theta)$. 
Since $F_m^{\prime} = -2G_m$ and the argument in the proof of Lemma  \ref{lem:decreasing-origin}, $G_m \in C^1 (\mathbb R)$ and hence $F_m \in C^2 (\mathbb{R})$. 
Hence $F_m$ admits the second-order Taylor expansion at $\theta = 0$. 
Furthermore $F_m^{\prime\prime} (0) = -2 G_m^{\prime}(0) > 0$ by Lemma  \ref{lem:decreasing-origin}. 
Finally we recall Theorem \ref{thm:SLLN}. 
Thus the assumptions of \cite[Theorem 5.23]{Vaart1998} are satisfied and Theorem \ref{thm:CLT} follows. 
\end{Rem}

\section{Proof of Theorem \ref{thm:LIL}}\label{sec:LIL}

We first give an outline of the proof. 
We follow the strategy of Barczy and P\'ales \cite[Section 5]{BarczyPales2023}. 
As in the case of Theorem \ref{thm:CLT}, we cannot apply the result of   \cite{BarczyPales2023} directly and need to modify several parts. 
We evaluate the normalized score along the scale of the law of the iterated logarithm and use the decomposition in \eqref{eq:RST} below into the leading fluctuation term $R_n$, the drift term $\gamma S_n$, and the remainder $T_n^{(\gamma)}$. 
We apply the Kolmogorov law of the iterated logarithm to $R_n$, the strong law of large numbers to $S_n$, and finally show that $T_n^{(\gamma)}$ is negligible.

Let $\phi(n) \coloneqq \sqrt{2n \log \log n}$. 
Let 
\[ R_n \coloneqq \frac{1}{\phi(n)} \sum_{i=1}^{n} D(X_i, 0),  \]
\[ S_n \coloneqq \frac{1}{n} \sum_{i=1}^{n} \partial_t D (X_i, 0), \]
and 
\[ T_n^{(\gamma)} \coloneqq \frac{1}{\phi(n)} \sum_{i=1}^{n} \left(D\left(X_i, \gamma \frac{\phi(n)}{n}\right) - D(X_i, 0) - \gamma \frac{\phi(n)}{n} \partial_t D (X_i, 0)\right). \]
Then  
\begin{equation}\label{eq:RST}  
R_n + \gamma S_n + T_n^{(\gamma)} = \frac{1}{\phi(n)} \sum_{i=1}^{n} D\left(X_i, \gamma \frac{\phi(n)}{n}\right). 
\end{equation}

Since $\displaystyle \max_{x,t \in \mathbb{R}}|D(x,t)| \le 1/2$, 
by the Kolmogorov law of the iterated logarithm,  
there exists an event $\Omega_4$ such that $P(\Omega_4) = 1$ and for every $\omega \in \Omega_4$, 
\begin{equation}\label{eq:lim-R}   
\limsup_{n \to \infty} R_n (\omega) = \sqrt{\frac{2m -1}{4m (m +1)}}.    
\end{equation}
By the strong law of large numbers, 
there exists an event $\Omega_5$ such that $P(\Omega_5) = 1$ and for every $\omega \in \Omega_5$, 
\begin{equation}\label{eq:lim-S}  
\lim_{n \to \infty} S_n (\omega) = -\frac{2m -1}{2(m + 1)}.    
\end{equation}

Let 
$\displaystyle \Omega_{6} \coloneqq \Omega_4 \cap \Omega_5 \cap \left(\bigcup_{N \ge 1} \bigcap_{n \ge N} \mathcal{A}_n \right)$. 
Then  
$P\left(\Omega_{6}\right) = 1$. 

By the uniform estimate \eqref{eq:non-as-estimate}, 
for every $\omega \in \Omega_{6}$ and every $\gamma \in \mathbb{R}$, 
\begin{equation}\label{eq:lim-T}   
\lim_{n \to \infty} T_{n}^{(\gamma)}(\omega) = 0. 
\end{equation}

Let $\sigma_{m} \coloneqq \sqrt{\dfrac{m + 1}{m (2m -1)}}$. 
Let $\omega \in \Omega_{6}$ and $\epsilon > 0$. 
Then  there exists $N_9 (\omega, \epsilon) \in \mathbb{N}$ such that for every $n \ge N_9 (\omega, \epsilon)$, 
$\hat{\theta}_n (\omega) < (\sigma_{m} + \epsilon) \dfrac{\phi(n)}{n}$ holds if and only if $\displaystyle \sum_{i=1}^{n} D\left(X_i (\omega), (\sigma_{m} + \epsilon)  \frac{\phi(n)}{n}\right) < 0$ holds. 
By \eqref{eq:RST}, this is equivalent to  
\begin{equation}\label{eq:RST-sum-neg}
R_n (\omega) + (\sigma_{m} + \epsilon)  S_n (\omega) + T_n^{(\sigma_{m} + \epsilon)}(\omega) < 0. 
\end{equation}
By \eqref{eq:lim-R}, \eqref{eq:lim-S} and \eqref{eq:lim-T}, 
there exists $N_{10} (\omega, \epsilon) \in \mathbb{N}$ such that for every $n \ge N_{10} (\omega, \epsilon)$, 
\eqref{eq:RST-sum-neg} holds.  
Hence,  for every $n \ge \max\{N_{9} (\omega, \epsilon), N_{10} (\omega, \epsilon)\}$, 
$\hat{\theta}_n (\omega) < (\sigma_{m} + \epsilon) \dfrac{\phi(n)}{n}$ holds
and hence, 
\[ \limsup_{n \to \infty}\frac{n \hat{\theta}_n (\omega)}{\phi(n)} \le \sigma_{m} + \epsilon.  \]
By letting $\epsilon \to 0$, 
\begin{equation}\label{eq:LIL-upper} 
\limsup_{n \to \infty}\frac{n \hat{\theta}_n (\omega)}{\phi(n)} \le \sigma_{m}.  
\end{equation}

We can show the lower bound in the same manner. 
There exists $N_{11} (\omega, \epsilon) \in \mathbb{N}$ such that for every $n \ge N_{11} (\omega, \epsilon)$, 
$\hat{\theta}_n (\omega) > (\sigma_{m} - \epsilon) \dfrac{\phi(n)}{n}$ holds if and only if $\displaystyle \sum_{i=1}^{n} D\left(X_i (\omega), (\sigma_{m} - \epsilon)  \frac{\phi(n)}{n}\right) > 0$ holds. 
By \eqref{eq:RST}, this is equivalent to  
\begin{equation}\label{eq:RST-sum-posi}
R_n (\omega) + (\sigma_{m} - \epsilon)  S_n (\omega) + T_n^{(\sigma_{m} - \epsilon)}(\omega) > 0. 
\end{equation}

By \eqref{eq:lim-R}, \eqref{eq:lim-S} and \eqref{eq:lim-T}, 
\eqref{eq:RST-sum-posi} holds for infinitely many $n$.  
Hence, 
$\hat{\theta}_n (\omega) > (\sigma_{m} - \epsilon) \dfrac{\phi(n)}{n}$ holds for infinitely many $n$, 
and hence, 
\[ \limsup_{n \to \infty}\frac{n \hat{\theta}_n (\omega)}{\phi(n)} \ge \sigma_{m} - \epsilon.  \]
By letting $\epsilon \to 0$, 
\begin{equation}\label{eq:LIL-lower}
\limsup_{n \to \infty}\frac{n \hat{\theta}_n (\omega)}{\phi(n)} \ge \sigma_{m}. 
\end{equation}

By \eqref{eq:LIL-upper} and \eqref{eq:LIL-lower}, 
\[ \limsup_{n \to \infty}\frac{n \hat{\theta}_n (\omega)}{\phi(n)} = \sigma_{m}. \]  

This completes the proof of Theorem \ref{thm:LIL}. 

\section{Proof of Theorem \ref{thm:BH}}\label{sec:BE}

We first give an outline of the proof. 
We prove Theorem \ref{thm:BH} by following the strategy of Bai and Fu \cite{BaiFu1987}. 
We first recall that $\left|P\left(\hat{\theta}_n > \epsilon\right) - P(D_n (\epsilon) > 0)\right|  \le 2P\left(\mathcal{A}_{n}^c \right)$ by the arguments of Section \ref{sec:CLT}. 

The first step is to show \eqref{eq:diff-MLE-D} below, which states that the probability of the localization event $P\left(\mathcal{A}_{n}^c \right)$ decays exponentially fast as $n \to \infty$. 
In Section \ref{sec:CLT}, we have seen that for $i=1,2$, $P(\mathcal{A}_{n,i}^c)$ decays exponentially fast. 
So it remains to control $P\left(\mathcal{A}_{n,3}^c \right)$.  
This is done by two large deviation estimates for $L_n (t)$. 
First, Lemma \ref{lem:interpolation-tail} below controls $\displaystyle \inf_{|t| \ge r} L_n (t)$ by using Lemmas \ref{lem:tilde-F-asymptotics} and \ref{lem:pre-interporation-tail} below. 
Second, Lemma \ref{lem:origin-small} below controls $L_n (0)$ by an exponential Chebyshev bound. 
Combining these bounds, we obtain \eqref{eq:diff-MLE-D}. 

After this, the problem reduces to estimating $P(D_n (\epsilon) > 0)$. 
We center $D(X_i, \epsilon)$ by its mean $G_{m}(\epsilon)$ and then apply two deviation inequalities in Lemmas \ref{lem:Petrov-1} and \ref{lem:Petrov-2} below with its variance $H_m (\epsilon)$. 
The Taylor expansions of $G_m$ and $H_m$ around $0$ in Lemma \ref{lem:GH-epsilon} below together with \eqref{eq:asymp-GH} identify the quadratic rate constant, which matches the upper and lower bounds in \eqref{eq:BH-upper} and \eqref{eq:BH-lower} respectively.

Recall the definition of $F_{m}$ in \eqref{eq:def-F-alpha}. 
Let 
\begin{equation*}\label{eq:def-tilde-F-m}
\widetilde{F}_{m} (t) \coloneqq \exp\left(F_{m}(t)\right). 
\end{equation*}

\begin{Lem}\label{lem:tilde-F-asymptotics}
\[ \lim_{t \to \infty} \frac{\widetilde{F}_{m} (t)}{t^2} = 1.  \]
\end{Lem}

\begin{proof}
The statement is equivalent to  
\begin{equation}\label{eq:F-asymptotics} 
\lim_{t \to \infty} F_{m}(t) - \log(1+t^2) = 0.  
\end{equation}
We see that 
\[ F_{m}(t) - \log(1+t^2) = c_{m} \int_{\mathbb R} \frac{\log(1+(x-t)^2) - \log(1+t^2)}{(1+x^2)^{m}}  dx  \]
and 
\[ \left|\log(1+(x-t)^2) - \log(1+t^2)\right| \le \log(2(1+x^2)). \]
Now we can apply the Lebesgue convergence theorem. 
\end{proof}

Let 
\begin{equation}\label{eq:def-lambda-m}
\lambda_{m} \coloneqq \frac{1}{2} \left(m - \frac{1}{2} \right) 
\end{equation}
and 
\begin{equation}\label{eq:def-delta-m-r}
\delta_m (r) \coloneqq \frac{F_m (r) - F_m (0)}{4}, \ r > 0. 
\end{equation}
These definitions will be used frequently not only in this section but also in the following section. 

The following corresponds to \cite[(3.15)]{BaiFu1987}. 

\begin{Lem}\label{lem:pre-interporation-tail}
Let $r > 0$. 
Assume that 
$0 < \delta < \min\left\{\lambda_m, \delta_m (r)  \right\}$.   
Then  there exists a positive constant  $c_{m,3}$ depending only on $m$ such that for every $t$ with $|t| > r$ and every $n \ge 1$, 
\[ P(L_n (t) \le F_{m}(0) + \delta) \le \exp\left(- \frac{\delta}{c_{m,3}} (F_{m}(t)-F_{m}(0) -2\delta) n \right). \]
\end{Lem}

The following proof is similar to the proof of Bernstein's inequality (\cite[Theorem 2.10]{BLM2013}). 
However the estimates are different, see \eqref{eq:Bernstein-type} below. 
The proof below is easier in the sense that there is no need to consider the Fenchel--Legendre transform. 

\begin{proof}
We assume that $t > r$. 
The proof is the same for the case that $t < -r$. 
We see that 
\[ P(L_n (t) \le F_{m}(0) + \delta) = P\left( \sum_{i=1}^{n} (F_{m}(t) - \log(1+ (X_i - t)^2))  \ge n(F_{m}(t) - F_{m}(0) - \delta) \right). \]
It holds that $F_{m}(t) - F_{m}(0) - \delta \ge  F_{m}(r) - F_{m}(0) - \delta > 0$ by Lemma \ref{lem:min-integral-stable}. 
By the exponential Chebyshev inequality, 
\[ P\left( \sum_{i=1}^{n} (F_{m}(t) - \log(1+ (X_i - t)^2))  \ge n(F_{m}(t) - F_{m}(0) - \delta) \right) \]
\[ \le \left(\exp\left( - \lambda (F_{m}(t) - F_{m}(0) - \delta) \right) E\left[ \exp\left( \lambda (F_{m}(t) - \log(1+ (X_1 - t)^2)) \right) \right] \right)^n \]
for every $\lambda > 0$. 

Assume that $0 < \lambda < m - \dfrac{1}{2}$. 
Then  
\[ E\left[ \exp\left( \lambda \left|F_{m}(t) - \log(1+ (X_1 - t)^2)\right| \right) \right] \le \exp(\lambda F_{m}(t)) E\left[ \left(1+(X_1 -t)^2 \right)^{\lambda}\right] < \infty. \]
Therefore, we can apply the Taylor expansion and obtain that 
\[ E\left[ \exp\left( \lambda (F_{m}(t) - \log(1+ (X_1 - t)^2)) \right) \right] = \sum_{k=0}^{\infty} \frac{\lambda^k}{k!} E\left[ \Psi(X_1, t)^k \right]. \]
where we let 
$\displaystyle \Psi(x,t) \coloneqq  \log \frac{\widetilde{F}_{m} (t)}{1+ (x - t)^2}$.   
Since $E\left[ \Psi(X_1, t)  \right] = 0$, 
\[  \sum_{k=0}^{\infty} \frac{\lambda^k}{k!} E\left[ \Psi(X_1, t) ^k \right] = 1 +  \sum_{k=2}^{\infty} \frac{\lambda^k}{k!} E\left[ \Psi(X_1, t)^k \right]. \]

By Lemma \ref{lem:tilde-F-asymptotics}, 
\[ c_{m,4} \coloneqq \sup_{x \le t/2, t \ge 0} \Psi(x,t) < \infty. \]
Hence, 
\begin{equation}\label{eq:Bernstein-type} 
E\left[ \Psi(X_1,t)^k \right] \le c_{m,4}^k P(X_1 \le t/2) + F_{m}(t)^k  P(X_1 > t/2). 
\end{equation} 
Since 
\[ P(X_1 > t/2) \le c_{m} \int_{t/2}^{\infty} x^{-2m} dx \le c_{m} 4^{m} t^{1-2m},  \]
we see that 
\[ E\left[ \Psi(X_1, t)^k \right] \le c_{m,4}^k + \min\left\{1, \frac{c_{m,5}}{t^{2m -1}} \right\} F_{m}(t)^k, \]
where we let $c_{m,5} \coloneqq c_{m} 4^{m}$. 
Hence, 
\[ \sum_{k=2}^{\infty} \frac{\lambda^k}{k!} E\left[ \Psi(X_1,t)^k \right]  \le \frac{\lambda^2}{2} \left( c_{m,4}^2  \exp(\lambda c_{m,4}) +  \min\left\{1, \frac{c_{m, 5}}{t^{2m -1}} \right\} F_{m}(t)^2 \exp(\lambda F_{m}(t)) \right). \]
Since $0 < \lambda < m - \dfrac{1}{2}$, 
\[ \lim_{t \to \infty}  \frac{F_{m}(t)^2 \exp(\lambda F_{m}(t)) }{t^{2m -1}} = 0, \]
and hence, 
\[ \sup_{t \ge 0}  \min\left\{1, \frac{c_{m,5}}{t^{2m -1}} \right\} F_{m}(t)^2 \exp(\lambda F_{m}(t)) < \infty.  \]

Recall \eqref{eq:def-lambda-m}.  
Then, 
for every $\lambda \in (0, \lambda_{m})$, 
\[ \sum_{k=2}^{\infty} \frac{\lambda^k}{k!} E\left[ \Psi(X_1,t)^k \right] \le \lambda^2 c_{m,6},  \]
where we let 
\[ c_{m,6} \coloneqq \frac{1}{2} \left(c_{m,4}^2  \exp(\lambda_{m} c_{m,4}) +  \sup_{t \ge 0} \min\left\{1, \frac{c_{m,5}}{t^{2m -1}} \right\} F_{m}(t)^2 \exp(\lambda_{m} F_{m}(t))  \right) < \infty. \]
We can assume that $c_{m,6} \ge 1$ because if $c_{m,6} < 1$, then we can replace $c_{m,6}$ with $c_{m,6} +1$. 

Therefore, 
for every $\displaystyle \lambda \in \left(0, \min\left\{\lambda_{m}, \delta_m (r) \right\} \right)$, 
\[ E\left[ \exp\left( \lambda (F_{m}(t) - \log(1+ (X_1 - t)^2)) \right) \right]  \le \exp(\lambda^2 c_{m,6}). \]

If we let $\lambda \coloneqq \delta / c_{m,6}$, then, $0 < \lambda < m - \dfrac{1}{2}$, and, 
\[ \exp\left( - \lambda (F_{m}(t) - F_{m}(0) - \delta) \right) E\left[ \exp\left( \lambda (F_{m}(t) - \log(1+ (X_1 - t)^2)) \right) \right] \]
\[ \le \exp\left(-\frac{\delta}{c_{m,6}} (F_{m}(t) - F_{m}(0) - 2\delta) \right). \]
Thus, the assertion holds for $c_{m,3} = c_{m,6}$. 
\end{proof}

Let $\displaystyle \lambda_{m}(r) \coloneqq \frac{1}{2}\min\left\{\lambda_m, \delta_m (r) \right\}$.

The following corresponds to \cite[(3.21)]{BaiFu1987}\footnote{There is a typo in \cite[(3.21)]{BaiFu1987}. The supremum in \cite[(3.21)]{BaiFu1987} should be the infimum.}. 

\begin{Lem}\label{lem:interpolation-tail}
Let $r> 0$. 
Assume that $\displaystyle 0 < \delta < \lambda_{m}(r)$. 
Then  there exists $N(r, \delta) \in \mathbb{N}$ such that for every $n \ge N(r, \delta)$, 
\[ P\left(\inf_{|t| \ge r} L_n (t) < F_{m}(0) + \delta \right) \le 2 \exp\left(- \frac{\delta^2}{ 8 c_{m,3}} n \right), \]
where $c_{m,3}$ is the constant appearing in Lemma \ref{lem:pre-interporation-tail}. 
\end{Lem}

We remark that $r > 0$ can be taken arbitrarily small. 
We first discretize $[r,\infty)$ by the Lipschitz continuity of $L_n(t)$ and then apply Lemma \ref{lem:pre-interporation-tail}. 

\begin{proof}
We show that 
\begin{equation}\label{eq:right-zero}
P\left(\inf_{t \ge r} L_n (t) < F_{m}(0) + \delta\right) \le  \exp\left(- \frac{\delta^2}{ 8 c_{m,3}} n \right). 
\end{equation}

Since $L_n^{\prime}(t) = -2D_n (t)$ and $|D_n (t)| \le 1/2$, 
it holds that 
\[ \left\{\inf_{t \ge r} L_n (t) < F_{m}(0) + \delta \right\} \subset \bigcup_{k \ge 1} \left\{ L_n (k\delta + r) < F_{m}(0) + 2\delta \right\}  \]
and hence, by Lemma \ref{lem:pre-interporation-tail}, 
\[ P\left(\inf_{t \ge r} L_n (t) < F_{m}(0) + \delta \right) \le \sum_{k=1}^{\infty} P\left(L_n (k\delta + r) < F_{m}(0) + 2\delta  \right) \]
\[ \le \sum_{k=1}^{\infty} \exp\left(- \frac{\delta}{c_{m,3}} (F_{m}\left( k\delta + r \right)-F_{m}(0) -4\delta) n \right) \]
\[ = \exp\left(- \frac{\delta}{c_{m,3}} (F_{m}(r)-F_{m}(0) -4\delta) n \right) \sum_{k=1}^{\infty} \exp\left(- \frac{\delta}{c_{m,3}} (F_{m}\left( k\delta + r \right)-F_{m}(r)) n \right) \]
\[ \le \exp\left(- \frac{\delta^2}{8c_{m,3}} n \right) \sum_{k=1}^{\infty} \exp\left(- \frac{\delta}{c_{m,3}} (F_{m}\left( k\delta + r \right)-F_{m}(r)) n \right). \]

By \eqref{eq:F-asymptotics}, 
there exists a positive constant $T_{m, r}$ such that for every $t > T_{m, r}$, $F_{m}(t) \ge F_{m}(r) +  \log t$. 
Hence, there exists $N_{T_{m,r}} \in \mathbb{N}$ such that for every $k > N_{T_{m,r}}$, $F_{m}(k \delta + r) \ge F_{m}(r) +  \log (k \delta + r)$. 
Since 
\[ \sum_{k=1}^{\infty} \exp\left(- \frac{\delta}{c_{m,3}} (F_{m}\left( k\delta + r \right)-F_{m}(r)) n \right) \]
\[ \le N_{T_{m,r}} \exp\left(- \frac{\delta}{c_{m,3}} (F_{m}\left( \delta + r \right)-F_{m}(r)) n \right) +  \sum_{k= N_{T_{m,r}} + 1}^{\infty} (k\delta + r)^{-n\delta / c_{m,3}}. \]
Hence, for large $n$,  
\[  \sum_{k=1}^{\infty} \exp\left(- \frac{\delta}{c_{m,3}} (F_{m}\left( k\delta + r \right)-F_{m}(r)) n \right) \le 1. \]
Thus \eqref{eq:right-zero} holds. 

The case that $t \le -r$ can be dealt with in the same manner. 
\end{proof}

The following corresponds to \cite[(3.25)]{BaiFu1987}\footnote{There is also a typo in \cite[(3.25)]{BaiFu1987}. ``$n^2$'' in the right hand side of the inequality in \cite[(3.25)]{BaiFu1987} should be ``$n \delta^2$''.}. 
Recall the definition of $\lambda_m$ in \eqref{eq:def-lambda-m}. 

\begin{Lem}\label{lem:origin-small}
There exists a positive constant $c_{m,7}$ depending only on $m$ such that for every $\delta \in (0, c_{m,7} \lambda_{m})$ and every $n \ge 1$, 
\[ P\left(L_n (0) \ge F_{m}(0) + \delta \right) \le \exp\left(-\frac{n \delta^2}{2c_{m,7}} \right). \] 
\end{Lem}

\begin{proof}
Assume that $0 < \lambda \le \lambda_{m}$. 
Then, 
by the exponential Chebyshev inequality, 
\[ P(L_n (0) \ge F_{m}(0) + \delta) \le \left( \exp(-\lambda \delta) E\left[ \exp(\lambda (\log(1+X_1^2) - F_{m}(0))) \right] \right)^n. \]

Since $E\left[\log(1+X_1^2) \right] = F_{m}(0)$, 
\[ E\left[ \exp(\lambda (\log(1+X_1^2) - F_{m}(0))) \right]  \le \exp\left( \frac{\lambda^2}{2} c_{m,7}  \right), \]
where we let 
\[ c_{m,7} \coloneqq E\left[ (\log(1+X_1^2) - F_{m}(0)))^2 \exp\left(\lambda_{m} \left|\log(1+X_1^2) - F_{m}(0))\right|  \right) \right]. \]
Now let $\lambda \coloneqq \delta/c_{m,7}$.  
\end{proof}

Let  
\[ c_{m,8} \coloneqq \dfrac{1}{2}\min\left\{\lambda_{m}(r_{m,2}/3), c_{m,7} \lambda_{m} \right\}.\]  
Let $\mathcal{B}_{n,1}$ be the event that $\displaystyle \inf_{|t| \ge r_{m,2}/3} L_n (t) < F_{m}(0) + c_{m,8}$. 
Let $\mathcal{B}_{n,2}$ be the event that $L_n (0) \ge F_{m}(0) + c_{m,8}$. 
Then  $\hat{\theta}_n \in [-r_{m,2}/2, r_{m,2}/2]$ on the event $\mathcal{B}_{n,1} \cap \mathcal{B}_{n,2}$. 
Therefore, 
$$ \mathcal{A}_{n,1} \cap \mathcal{A}_{n,2} \cap \mathcal{B}_{n,1} \cap \mathcal{B}_{n,2} \subset \mathcal{A}_{n}. $$
By Lemma  \ref{lem:lv2-Azuma}, Lemma \ref{lem:lv3-Azuma}, Lemma \ref{lem:interpolation-tail}, and Lemma \ref{lem:origin-small}, 
there exist constants $c_{m,9}, c_{m,10}$ depending only on $m$ such that for every $n \ge 1$, 
\[ P(\mathcal{A}_n^c) \le P(\mathcal{A}_{n,1}^c) + P(\mathcal{A}_{n,2}^c) + P(\mathcal{B}_{n,1}^c)  +P(\mathcal{B}_{n,2}^c) \le c_{m,9} \exp(- c_{m,10} n). \]

For $\epsilon \in (0, r_{m,2}/4)$, 
\[ P(\{\hat{\theta}_n > \epsilon \} \cap \mathcal{A}_n) = P(\{D_n (\epsilon) > 0 \} \cap \mathcal{A}_n)  \]
and hence, 
\begin{equation}\label{eq:diff-MLE-D}
\left|P\left(\hat{\theta}_n > \epsilon\right) - P(D_n (\epsilon) > 0)\right| \le 2P(\mathcal{A}_n^c) \le  2c_{m,9} \exp(- c_{m,10} n), \ n \ge 1. 
\end{equation}

Let 
\[ H_{m}(\epsilon) \coloneqq \textup{Var}(D(X_1,\epsilon)) = E\left[ D(X_1,\epsilon)^2 \right] - G_{m}(\epsilon)^2.  \]

\begin{Lem}\label{lem:GH-epsilon}
It holds that \\
(1) $\displaystyle G_{m}(\epsilon) = -\frac{2m-1}{2(m + 1)} \epsilon + O(\epsilon^2),  \ \epsilon \to +0$. \\
(2) $\displaystyle H_{m}(\epsilon) = \frac{2m -1}{4m (m + 1)} + O(\epsilon), \ \epsilon \to +0$. 
\end{Lem}

\begin{proof}
(1) By \eqref{eq:non-as-estimate},  
\begin{equation}\label{eq:diff-epsilon}
\left|D(X_1, \epsilon) - D(X_1,0) - \epsilon \partial_t D (X_1, 0)\right| \le C_1 \epsilon^2. 
\end{equation}
By \eqref{eq:mean-exp-D-zero} and \eqref{eq:mean-partial-D}, 
\[ E\left[ D(X_1,0) \right] = 0, E\left[\partial_t D (X_1, 0)\right] = - \frac{2m - 1}{2(m + 1)}.  \]
The estimate follows from these equalities and \eqref{eq:diff-epsilon}.

(2) By \eqref{eq:diff-epsilon}, 
there exists a positive constant $C_2$ such that for every $\epsilon \in (0,1)$, 
\begin{equation}\label{eq:diff-epsilon-squared}
\left|D(X_1, \epsilon)^2 - D(X_1,0)^2 - 2 \epsilon D(X_1,0) \partial_t D (X_1, 0)\right| \le C_2 \epsilon^2. 
\end{equation}
Since $D(X_1,0)$ and $\partial_t D (X_1, 0)$ are bounded, 
$D(X_1,0) \partial_t D (X_1, 0)$ is also bounded, and in particular, is integrable. 
By \eqref{eq:squared-exp-D}, 
\[ H_{m}(0) = E\left[ D(X_1,0)^2 \right] =  \frac{2m -1}{4m (m + 1)}. \]
The estimate follows from this equality and \eqref{eq:diff-epsilon-squared}.
\end{proof}

We show (i) of Theorem \ref{thm:BH}. 
We consider the asymptotics of $P(D_n (\epsilon) > 0)$ as $n \to \infty$. 

We first give the upper estimate. 
We remark that $|D(X_i, \epsilon) -   G_{m}(\epsilon)| \le \frac{1}{2} - G_{m}(\epsilon)$ and by Lemma \ref{lem:GH-epsilon}, 
\[ \lim_{\epsilon \to +0}  G_{m}(\epsilon) \left(\frac{1}{2} - G_{m}(\epsilon)\right) = 0,  \]
and, 
\[ \lim_{\epsilon \to +0}  H_{m}(\epsilon) = H_{m}(0) > 0.\] 
Hence, there exists a constant $\epsilon_{m,1} > 0$ depending only on $m$ such that 
for every $\epsilon \in (0, \epsilon_{m,1})$, 
\[ \left| D(X_i, \epsilon) -   G_{m}(\epsilon)\right| \le H_{m}(\epsilon). \]

\begin{Lem}[{Petrov \cite[Lemma 7.1]{Petrov1995}\footnote{The statement is a little different from \cite[Lemma 1]{BaiFu1987}.  In \cite[Lemma 1]{BaiFu1987}, this assertion holds for large $n$, but this is valid for every $n \ge 1$.}}]\label{lem:Petrov-1}
Let $Z_i, i \ge 1$, be i.i.d. random variables such that $|Z_1| \le M$, $P$-a.s., $E[Z_1] = 0$, and $\sigma^2 \coloneqq \textup{Var}(Z_1) > 0$.  
Then, for every $n \ge 1$ and every $x \in [0, \sigma^2/M]$, 
\[ P\left(\sum_{i=1}^{n} Z_i \ge nx \right) \le \exp\left(-\frac{n x^2}{2\sigma^2} \left(1 - \frac{Mx}{2\sigma^2}\right) \right). \]
\end{Lem}
 
By this lemma, 
it holds that for every $\epsilon \in (0, \epsilon_{m,1})$ and every  $n \ge 1$, 
\begin{align}\label{eq:Petrov-1} 
P(D_n (\epsilon) > 0) &= P\left( \sum_{i=1}^{n} D(X_i, \epsilon) -   G_{m}(\epsilon) > -nG_{m}(\epsilon) \right) \notag\\
&\le \exp\left( -\frac{n G_{m}(\epsilon)^2}{2H_{m}(\epsilon)} \left(1+ \frac{G_{m}(\epsilon)}{2 H_{m} (\epsilon)} \right)  \right). 
\end{align}

By Lemma \ref{lem:GH-epsilon}, 
\begin{equation}\label{eq:asymp-GH} 
\frac{G_{m}(\epsilon)^2}{H_{m}(\epsilon)} \left(1+ \frac{G_{m}(\epsilon)}{2 H_{m} (\epsilon)} \right) \sim \frac{m (2m-1)}{m + 1} \epsilon^2, \ \epsilon \to +0, 
\end{equation}
in particular, 
\[ \lim_{\epsilon \to +0} \frac{G_{m}(\epsilon)^2}{H_{m}(\epsilon)} \left(1+ \frac{G_{m}(\epsilon)}{2 H_{m} (\epsilon)} \right) = 0. \]
By this, \eqref{eq:Petrov-1}, and \eqref{eq:diff-MLE-D}, 
it holds that there exists $\epsilon_{m,2} > 0$ such that for every $\epsilon \in (0, \epsilon_{m,2})$, there exists $N_{\epsilon}$ such that for every $n \ge N_{\epsilon}$, 
\[ P\left(\hat{\theta}_n > \epsilon\right) \le 2 \exp\left( -\frac{n G_{m}(\epsilon)^2}{2H_{m}(\epsilon)} \left(1+ \frac{G_{m}(\epsilon)}{2 H_{m} (\epsilon)} \right)  \right). \]

Hence, for every $\epsilon \in (0, \epsilon_{m,2})$, 
\[ \limsup_{n \to \infty} \frac{\log P\left(\hat{\theta}_n > \epsilon\right)}{n} \le -\frac{G_{m}(\epsilon)^2}{2H_{m}(\epsilon)} \left(1+ \frac{G_{m}(\epsilon)}{2 H_{m} (\epsilon)} \right). \]
By this, Lemma \ref{lem:GH-epsilon}, and \eqref{eq:asymp-GH}, 
\[ \limsup_{\epsilon \to +0} \frac{1}{\epsilon^2} \left( \limsup_{n \to \infty} \frac{\log P\left(\hat{\theta}_n  > \epsilon\right)}{n} \right) \le -\frac{m (2m-1)}{2(m + 1)}.  \]

The same argument is applicable to $P\left(\hat{\theta}_n  < -\epsilon\right)$ and we obtain \eqref{eq:BH-upper}.  

We next give the lower estimate \eqref{eq:BH-lower}. 
By Lemma \ref{lem:GH-epsilon},
\[ \lim_{\epsilon \to +0} G_{m}(\epsilon) = 0 \textup{  and } \lim_{\epsilon \to +0} H_{m}(\epsilon) = E\left[D(X_1,0)^2 \right] > 0. \] 

\begin{Lem}[{Petrov \cite[Lemma 7.2]{Petrov1995}\footnote{The statement is a little different from \cite[Lemma 7.2]{Petrov1995}, however, we can show this assertion in the same manner as in the proof of \cite[Lemma 7.2]{Petrov1995}.}}]\label{lem:Petrov-2}
Let $Z_i, i \ge 1$, be i.i.d. random variables such that $|Z_1| \le M$, $P$-a.s., $E[Z_1] = 0$, and $\sigma^2 \coloneqq \textup{Var}(Z_1) > 0$.  
Then, for every $\eta > 0$, there exists $r > 0$ such that for every $x \in [0, r]$, there exists $N$ such that for every $n \ge N$, 
\[ P\left(\sum_{i=1}^{n} Z_i \ge nx \right) \ge \exp\left(-\frac{n x^2}{2\sigma^2} \left(1 + \eta \right) \right). \]
\end{Lem}

By this lemma, for every $\eta > 0$, there exists $\epsilon_{\eta} > 0$ depending on $m$ and $\eta$ such that for every $\epsilon \in (0, \epsilon_{\eta})$, 
there exists $N_{\eta, \epsilon, 1}  \in \mathbb{N}$ such that for every $n \ge N_{\eta, \epsilon, 1}$, 
\begin{equation}\label{eq:Petrov-2}  
P(D_n (\epsilon) > 0) \ge \exp\left( -\frac{n G_{m}(\epsilon)^2}{2H_{m}(\epsilon)} (1+\eta) \right). 
\end{equation}

In the same manner as in the upper bound,  
it holds that there exists $\epsilon_{\eta, 2} > 0$ depending on $\eta$ such that for every $\epsilon \in (0, \epsilon_{\eta,2})$, there exists $N_{\eta, \epsilon,2} \in \mathbb{N}$ such that for every $n \ge N_{\eta, \epsilon, 2}$, 
\[ P\left(\hat{\theta}_n > \epsilon\right) \ge \frac{1}{2} \exp\left( -\frac{n G_{m}(\epsilon)^2}{2H_{m}(\epsilon)} (1+\eta) \right). \]

Hence, for every $\epsilon \in (0, \epsilon_{\eta,2})$, 
\[ \liminf_{n \to \infty} \frac{\log P\left(\hat{\theta}_n > \epsilon\right)}{n} \ge -\frac{G_{m}(\epsilon)^2}{2H_{m}(\epsilon)} (1+\eta).  \]
By this and Lemma \ref{lem:GH-epsilon},  letting $\eta \to +0$, 
\[ \liminf_{\epsilon \to +0} \frac{1}{\epsilon^2} \left( \liminf_{n \to \infty} \frac{\log P\left(\hat{\theta}_n > \epsilon\right)}{n} \right) \ge -\frac{m (2m-1)}{2(m + 1)}.  \]
 
The same argument is applicable to $P\left(\hat{\theta}_n  < -\epsilon\right)$ and we obtain \eqref{eq:BH-lower}.  
Thus the proof of (i) of Theorem \ref{thm:BH} is completed. 

Now we show (ii)  of Theorem \ref{thm:BH}, but the proof is almost identical to the proof of (i).

By \eqref{eq:Petrov-1}, 
it holds that for large  $n$, 
\[ P(D_n (\epsilon/a_n) > 0) \le \exp\left( -\frac{n G_{m}(\epsilon/a_n)^2}{2H_{m}(\epsilon/a_n)} \left(1+ \frac{G_{m}(\epsilon/a_n)}{2 H_{m} (\epsilon/a_n)} \right)  \right). \] 
By Lemma \ref{lem:GH-epsilon}, 
\[ \lim_{n \to \infty} a_n^2 \frac{G_{m}(\epsilon/a_n)^2}{H_{m}(\epsilon/a_n)} \left(1+ \frac{G_{m}(\epsilon/a_n)}{2 H_{m} (\epsilon/a_n)} \right) = \frac{m (2m-1)}{m + 1}. \]
Therefore, we obtain that 
\begin{equation}\label{eq:MDP-wts-upper}
\limsup_{n \to \infty} \frac{\log P\left(D_n (\epsilon/a_n)  > 0 \right)}{n/a_n^2} \le -\frac{m (2m-1)}{2(m + 1)}\epsilon^2. 
\end{equation}

By \eqref{eq:Petrov-2} and Lemma \ref{lem:GH-epsilon}, 
we obtain that 
\begin{equation}\label{eq:MDP-wts-lower}
\liminf_{n \to \infty} \frac{\log P\left(D_n (\epsilon/a_n)  > 0 \right)}{n/a_n^2} \ge -\frac{m (2m-1)}{2(m + 1)}\epsilon^2. 
\end{equation}

\eqref{eq:MDP-wts-upper} and \eqref{eq:MDP-wts-lower} imply that 
\begin{equation*}\label{eq:MDP-wts}
\lim_{n \to \infty} \frac{\log P\left(D_n (\epsilon/a_n)  > 0 \right)}{n/a_n^2} = -\frac{m (2m-1)}{2(m + 1)}\epsilon^2. 
\end{equation*} 

By this and \eqref{eq:diff-MLE-D}, 
\[  \lim_{n \to \infty} \frac{\log P\left(\hat{\theta}_n  > \epsilon/a_n \right)}{n/a_n^2} = -\frac{m (2m-1)}{2(m + 1)}\epsilon^2. \]

$P\left(\hat{\theta}_n  < -\epsilon/a_n \right)$ can be dealt with in the same manner. 
Thus the proof of (ii) of Theorem \ref{thm:BH} is completed. 

\begin{Rem}
(i) Let $K(\cdot | \cdot)$ be the Kullback-Leibler divergence. 
Then, by computations, 
\[ K\left( \textup{PVII}_{m} (\theta_1, 1) | \textup{PVII}_{m} (\theta_2, 1) \right) = m(F_m (\theta_1 - \theta_2) - F_m (0)). \]
Let 
\[ b(\epsilon, \theta) \coloneqq \inf\left\{ K\left( \textup{PVII}_{m} (\theta^{\prime}, 1) | \textup{PVII}_{m} (\theta, 1) \right) \middle| |\theta^{\prime} - \theta| > \epsilon\right\}. \]
Since $F_m$ is symmetric and $t \mapsto F_m (|t|)$ is increasing, 
$b(\epsilon, \theta) = m(F_m (\epsilon) - F_m (0))$. 
Since $F_m^{\prime} = -2G_m$, 
\[ \lim_{\epsilon \to +0} \frac{b(\epsilon, \theta)}{\epsilon^2} = \frac{m (2m-1)}{2(m + 1)} = \frac{1}{I(\theta)}.  \]
(ii) For the case that $m=1$, the Bahadur efficiency for the {\it joint} estimation of the location and the scale is established in \cite[Theorem 4]{AOO2022AISM} when both the location and the scale are unknown. 
Recently, Akahira \cite{Akahira2025} showed that the MLE of the location parameter is first order large deviation efficient, which implies the Bahadur efficiency.\\ 
(iii) Gao \cite{Gao2001} obtained moderate deviation results for the maximum likelihood estimator in a more general framework under certain regular conditions. 
Our model does not satisfy the conditions because the likelihood equation $D_n (t) = 0$ has multiple roots. 
\end{Rem}

\section{Proof of Theorem \ref{thm:integrability}}\label{sec:integrability}

We first give an outline of the proof. 
We estimate the tail probability of $\hat{\theta}_n$ by comparing $L_n (t)$ at large $|t|$ with its expectation at the true parameter $t=0$, which is $F_m (0)$. 
The probability of the event $\{\hat{\theta}_n > r\}$ is controlled by the decomposition \eqref{eq:tail-fund} below. 
More specifically, if $\hat{\theta}_n > r$, then either $L_n (t)$ becomes unexpectedly small somewhere on $[r,\infty)$, or $L_n (0)$ becomes unexpectedly large.

In order to bound this probability, we control the two terms in  \eqref{eq:tail-fund} separately by modifying the assertions in Section \ref{sec:BE}. 
Lemma \ref{lem:pre-interporation-tail-2} below gives a polynomial-type estimate for the lower tail of $L_n (t)$ at a fixed large $t$ by the exponential Chebyshev inequality. 
Then Lemma \ref{lem:interpolation-tail-2} below upgrades this pointwise estimate to the whole tail region $[r, \infty)$ and controls the first term of  \eqref{eq:tail-fund}, by discretizing the region and using the Lipschitz bound for $L_n^{\prime}(t)$. 
Finally Lemma \ref{lem:origin-small-2} below provides an upper-tail bound for $L_n (0)$ and controls the second term of  \eqref{eq:tail-fund}.

By symmetry, we deal with $P\left(\hat{\theta}_n > r\right)$. 
We see that for every $r > 0$ and every $\delta > 0$, 
\begin{equation}\label{eq:tail-fund} 
P\left(\hat{\theta}_n > r\right) \le P\left( \inf_{t \ge r} L_n (t) < F_{m}(0) + \delta  \right) + P\left(L_n (0) \ge F_{m}(0) + \delta \right). 
\end{equation}

First, we give a lemma similar to Lemma \ref{lem:pre-interporation-tail}. 
The proof differs in part.  
Recall the definition of $\lambda_m$ in  \eqref{eq:def-lambda-m}. 

\begin{Lem}\label{lem:pre-interporation-tail-2}
There exist two constants  $r_{m,3}$ and $c_{m,11}$ such that for every $t \ge r_{m,3}$ and every $n \ge 1$, 
\[ P\left(L_n (t) \le F_{m}(0) + \frac{F_{m}(t) - F_{m}(0)}{2} \right) \le c_{m,11} t^{-n\lambda_{m}}. \]
\end{Lem}

\begin{proof}
As in Lemma \ref{lem:pre-interporation-tail}, by the exponential Chebyshev inequality, 
\[ P\left(L_n (t) \le F_{m}(0) +  \frac{F_{m}(t) - F_{m}(0)}{2}\right) \]
\[ = P\left( \sum_{i=1}^{n} (F_{m}(t) - \log(1+ (X_i - t)^2))  \ge n\left(F_{m}(t) - F_{m}(0) -  \frac{F_{m}(t) - F_{m}(0)}{2}\right) \right) \]
\[ \le \left(\exp\left(2\lambda_{m} \left(F_{m}(0) +  \frac{F_{m}(t) - F_{m}(0)}{2}\right) \right) E\left[ \left(1+ (X_1 - t)^2\right)^{-2\lambda_{m}} \right] \right)^n. \]

It holds that 
\[ E\left[ \left(1+ (X_1 - t)^2 \right)^{-2\lambda_{m}} \right] \]
\[ = E\left[ \left(1+ (X_1 - t)^2\right)^{-2\lambda_{m}}, \ X_1 \ge t/2 \right]  + E\left[ \left(1+ (X_1 - t)^2 \right)^{-2\lambda_{m}}, \ X_1 < t/2  \right]   \]
\[ \le P(X_1 \ge t/2) + \left(1+ \frac{t^2}{4} \right)^{-\lambda_{m}} = O(t^{1-2m}), \ \ t \to \infty.  \]

By \eqref{eq:F-asymptotics}, 
\[ \exp\left(2\lambda_{m} \left(F_{m}(0) +  \frac{F_{m}(t) - F_{m}(0)}{2}\right) \right) = O\left(t^{8\lambda_{m}/3}\right), \ \ t \to \infty. \]

Therefore, 
\[ \exp\left(2\lambda_{m} \left(F_{m}(0) +  \frac{F_{m}(t) - F_{m}(0)}{2}\right) \right) E\left[ \left(1+ (X_1 - t)^2 \right)^{-2\lambda_{m}} \right] = O(t^{-4\lambda_{m}/3}), \ t \to \infty.  \]
This completes the proof. 
\end{proof}

Next, we give a lemma similar to Lemma \ref{lem:interpolation-tail}. 
The proof is also similar. 
Recall the definition of $\delta_m (r)$ in \eqref{eq:def-delta-m-r}. 

\begin{Lem}\label{lem:interpolation-tail-2}
There exist two positive constants  $r_{m,4}$ and $c_{m,12}$ and $N_{m,1} \in \mathbb{N}$ depending only on $m$ such that 
for every $r \ge r_{m,4}$ and every $n \ge N_{m,1}$, 
\[ P\left(\inf_{t \ge r} L_n (t) < F_{m}(0) + \delta_m (r) \right) \le c_{m,12} r^{-n\lambda_{m}/2}, \]
\end{Lem}

\begin{proof}
Since $|L_n^{\prime}(t)| \le 1$, 
\[ \left\{\inf_{t \ge r} L_n (t) < F_{m}(0) + \delta_m (r)  \right\} \subset \bigcup_{k \ge 1} \left\{ L_n (k\delta_m (r) + r) < F_{m}(0) + 2\delta_m (r) \right\}  \]
and hence, by Lemma \ref{lem:pre-interporation-tail-2}, 
for every $n \ge 2/\lambda_{m}$ and every $r \ge r_{m,4}$, 
\[ P\left(\inf_{t \ge r} L_n (t) < F_{m}(0) + \delta_m (r) \right) \le \sum_{k=1}^{\infty} P\left(L_n (k\delta_m (r) + r) < F_{m}(0) + 2\delta_m (r)  \right) \]
\[  \le \sum_{k=1}^{\infty} P\left(L_n (k\delta_m (r) + r) < F_{m}(0) +  \frac{F_{m}(k\delta_m (r) + r) - F_{m}(0)}{2} \right) \]
\[ \le c_{m,11} \sum_{k=1}^{\infty} (k\delta_m (r) + r)^{-n\lambda_{m}} \le c_{m,11} \delta_m (r) \int_{r}^{\infty} x^{-n\lambda_{m}} dx \le c_{m,11} \delta_m (r) r^{1-n\lambda_{m}}.   \]
By \eqref{eq:F-asymptotics},  
$\delta_m (r) = O(\log r), r \to \infty$ and we have the assertion. 
\end{proof}

Finally, we give a lemma similar to Lemma \ref{lem:origin-small}. 

\begin{Lem}\label{lem:origin-small-2}
There exist positive constants $r_{m,5}$ and  $c_{m,13}$ such that for every $r > r_{m,5}$ and every $n \ge 1$, 
\[ P\left(L_n (0) \ge F_{m}(0) + \delta_m (r) \right) \le r^{-n c_{m,13}}. \] 
\end{Lem}

\begin{proof}
Let $ c_{m,7}$ be the constant as in the proof of Lemma \ref{lem:origin-small}.  
Assume that $0 < \lambda \le \lambda_{m}$. 
Then, 
by the exponential Chebyshev inequality, 
for every $n \ge 1$, 
\[ P(L_n (0) \ge F_{m}(0) + \delta_m (r)) \le \left( \exp\left(-\lambda \delta_m (r) + \frac{\lambda^2}{2} c_{m,7}  \right) \right)^n.  \]
By \eqref{eq:F-asymptotics}, 
there exists a positive constant $r_{m,5}$ such that for every $r > r_{m,5}$, $2c_{m,7} \le \log r \le \delta_m (r)$.  
Let $\lambda_{m}^{\prime} \coloneqq \min\{1, \lambda_{m}\}$. 
Thus, for every $r > r_{m,5}$ and every $n \ge 1$, 
\[ P(L_n (0) \ge F_{m}(0) + \delta_m (r)) \le \left( \exp\left( \left(-\lambda_{m}^{\prime} + \frac{(\lambda_{m}^{\prime})^2}{4}\right) \delta_m (r)  \right) \right)^n \le r^{-n c_{m,13}},   \]
where we let $c_{m,13} \coloneqq \lambda_{m}^{\prime} - \dfrac{(\lambda_{m}^{\prime})^2}{4} > 0$. 
\end{proof}

By applying \eqref{eq:tail-fund} to $\delta = \delta_m (r)$, by Lemma \ref{lem:interpolation-tail-2} and Lemma \ref{lem:origin-small-2},
there exist positive constants $r_{m}$ and  $N_{m} \in \mathbb{N}$ depending only on $m$ such that 
for every $r \ge r_{m}$ and every $n \ge N_{m}$, 
\[ P\left( \hat{\theta}_n  > r \right) \le r^{-c_{m,13} n}. \]
$P\left(\hat{\theta}_n < -r\right)$ can be dealt with in the same manner, and we obtain Theorem \ref{thm:integrability}.

\section{Proof of Theorem \ref{thm:AE}}\label{sec:AE}

We first give an outline of the proof. 
The proof is a uniform-integrability argument based on the asymptotic normality (Theorem \ref{thm:CLT}) and the Bahadur efficiency (two estimates used in the proof of Theorem \ref{thm:BH}) and the tail bound (Theorem \ref{thm:integrability}).
We show the lower and upper bounds separately.
The lower bound is an easy consequence of Theorem \ref{thm:CLT}. 

In order to obtain the upper bound, it suffices to show that the family $\left\{ \left(\sqrt{n} \hat{\theta}_n\right)^2 \right\}_n$ is uniformly integrable, which is reduced to \eqref{eq:AE-diff} below.
The key step is to rewrite the tail contribution as an integral of tail probabilities as in \eqref{eq:integral-AE-diff} and then split the integral into three regimes as in \eqref{eq:split-integral-into-three}. 
For the small-to-moderate region and the intermediate region, we use \eqref{eq:diff-MLE-D} and \eqref{eq:Petrov-1} used in the proof of Theorem \ref{thm:BH} respectively. 
For the far tail region, we apply Theorem \ref{thm:integrability}. 

For $M > 0$, 
let $\phi_M (x) \coloneqq \min\{x^2, M^2\}$. 
This is bounded and continuous on $\mathbb{R}$. 
Recall that $\varphi_{m}$ is the density function of the distribution $\displaystyle N\left(0,\dfrac{m + 1}{m (2m -1)} \right)$. 
By Theorem \ref{thm:CLT}, 
\begin{equation}\label{eq:AE-1st} 
\lim_{n \to \infty} E\left[ \phi_M \left(\sqrt{n} \hat{\theta}_n \right) \right]  = \int_{\mathbb R} \phi_M (x) \varphi_{m} (x) dx. 
\end{equation}

Since $x^2 \ge \phi_M (x)$, 
\[ \liminf_{n \to \infty} n E\left[ \left(\hat{\theta}_n\right)^2\right] \ge \int_{\mathbb R} \phi_M (x) \varphi_{m} (x) dx.  \]
By the monotone convergence theorem, 
\begin{equation}\label{eq:AE-inf} 
\liminf_{n \to \infty} n E\left[ \left(\hat{\theta}_n\right)^2\right] \ge \int_{\mathbb R} x^2 \varphi_{m} (x) dx = \frac{m + 1}{m (2m -1)} . 
\end{equation}
 
We will show that 
\begin{equation}\label{eq:AE-sup} 
\limsup_{n \to \infty} n E\left[ \left(\hat{\theta}_n\right)^2\right] \le \frac{m + 1}{m (2m -1)}.  
\end{equation}

By \eqref{eq:AE-1st} and the monotone convergence theorem, 
\[ \lim_{M \to \infty} \lim_{n \to \infty} E\left[ \phi_M \left(\sqrt{n} \hat{\theta}_n \right) \right]  = \int_{\mathbb R} x^2 \varphi_{m} (x) dx = \frac{m + 1}{m (2m -1)}. \]

Hence it suffices to show that
\begin{equation}\label{eq:AE-diff}
\limsup_{M \to \infty} \left(\limsup_{n \to \infty} E\left[ \left( \sqrt{n} \hat{\theta}_n \right)^2 - \phi_M \left(\sqrt{n} \hat{\theta}_n \right) \right]\right) = 0.
\end{equation}

By Fubini's theorem for non-negative measurable functions and the change of variables $t= \sqrt{s}$, 
we obtain that 
\begin{align}\label{eq:integral-AE-diff} 
E\left[ \left( \sqrt{n} \hat{\theta}_n \right)^2 - \phi_M \left(\sqrt{n} \hat{\theta}_n \right) \right] &= E\left[ \left( \sqrt{n} \hat{\theta}_n \right)^2 - M^2, \ \left|  \sqrt{n} \hat{\theta}_n \right| \ge M \right] \notag\\
&= 2\int_{M}^{\infty} t P\left( \left|  \sqrt{n} \hat{\theta}_n \right| > t \right) dt \notag\\
&= 2n \int_{M/\sqrt{n}}^{\infty} s P\left( \left| \hat{\theta}_n \right| > s \right) ds.  
\end{align}

By symmetry, we consider $P\left( \hat{\theta}_n  > s \right)$.  
By \eqref{eq:asymp-GH}, 
there exists $\epsilon_{m,3} \in (0, r_m)$ such that for every $\epsilon \in (0, 2\epsilon_{m,3})$, 
\begin{equation}\label{eq:G-over-H} 
\frac{G_{m}(\epsilon)^2}{H_{m}(\epsilon)} \left(1+ \frac{G_{m}(\epsilon)}{2 H_{m} (\epsilon)} \right) \ge \frac{m(2m-1)}{4(m+ 1)} \epsilon^2.  
\end{equation}
 
 Now we decompose the last integral in \eqref{eq:integral-AE-diff} into three parts: 
 \begin{equation}\label{eq:split-integral-into-three}  
 \int_{M/\sqrt{n}}^{\infty} =  \int_{M/\sqrt{n}}^{\epsilon_{m,3}} + \int_{\epsilon_{m,3}}^{r_{m}+1} + \int_{r_{m}+1}^{\infty},  
 \end{equation} 
 where $r_{m}$ is the constant in Theorem \ref{thm:integrability}.

By \eqref{eq:G-over-H}, \eqref{eq:Petrov-1}, and \eqref{eq:diff-MLE-D}, 
there exist two positive constants $c_{m,14}, c_{m,15}$ and $N_{m,2} \in \mathbb{N}$ depending only on $m$ such that 
for every $n \ge N_{m,2}$ and $s \in (0, 2\epsilon_{m,3})$, 
\begin{equation}\label{eq:BH-quantitative} 
P\left(\hat{\theta}_n > s\right) \le \exp\left(- \frac{m (2m -1)}{4(m + 1)} s^2 n\right) + c_{m,14} \exp(- c_{m,15} n).  
\end{equation}
Therefore, for $n \ge N_{m,2}$, 
\[ 2n \int_{M/\sqrt{n}}^{\epsilon_{m,3}} s P\left( \hat{\theta}_n  > s\right)  ds \]
\[ \le  \int_{M/\sqrt{n}}^{\epsilon_{m,3}} 2ns \exp\left(- \frac{m (2m -1)}{4(m + 1)} n s^2 \right) ds + n \epsilon_{m,3}^2 c_{m,14} \exp(- c_{m,15} n) \]
\[ \le \frac{4(m + 1)}{m (2m -1)}  \exp\left(- \frac{m (2m -1)}{4(m + 1)} M^2 \right)  + n \epsilon_{m,3}^2 c_{m,14} \exp(- c_{m,15} n). \]
Hence, 
\begin{equation}\label{eq:1st-integral-estimate} 
\limsup_{n \to \infty} 2n \int_{M/\sqrt{n}}^{\epsilon_{m,3}} s P\left( \hat{\theta}_n  > s\right)  ds \le \frac{4(m + 1)}{m (2m -1)}  \exp\left(- \frac{m (2m -1)}{4(m + 1)} M^2 \right). 
\end{equation} 

Since 
\[ 2n \int_{\epsilon_{m,3}}^{r_{m} + 1} s P\left( \hat{\theta}_n  > s\right)  ds \le 2 (r_{m} + 1)^2 n P\left( \hat{\theta}_n  > \epsilon_{m,3} \right), \]
by applying \eqref{eq:BH-quantitative} to $s = \epsilon_{m,3}$, 
\begin{equation}\label{eq:2nd-integral-estimate}  
\limsup_{n \to \infty} 2n \int_{\epsilon_{m,3}}^{r_{m} + 1} s P\left( \hat{\theta}_n  > s\right) ds = 0. 
\end{equation}  

By Theorem \ref{thm:integrability}, 
for large $n$, 
\[ n \int_{r_{m} + 1}^{\infty}  s P\left( \hat{\theta}_n  > s\right)  ds \le  n \int_{r_{m} + 1}^{\infty}  s^{1-nc_{m}} ds \le \frac{n}{c_{m} n - 2} (r_{m} + 1)^{2-n c_{m}}.   \]
Hence, 
\begin{equation}\label{eq:3rd-integral-estimate}  
\limsup_{n \to \infty} 2n \int_{r_{m} + 1}^{\infty} s P\left( \hat{\theta}_n  > s\right)  ds = 0. 
\end{equation}  

By \eqref{eq:1st-integral-estimate}, \eqref{eq:2nd-integral-estimate}, and \eqref{eq:3rd-integral-estimate}, 
\[ \limsup_{n \to \infty} 2n \int_{M/\sqrt{n}}^{\infty} s P\left( \hat{\theta}_n  > s\right)  ds \le \frac{4(m + 1)}{m (2m -1)}  \exp\left(- \frac{m (2m -1)}{4(m + 1)} M^2 \right). \]
The same estimate holds for $P\left( \hat{\theta}_n  < -s \right)$. 
Since the right hand side converges to $0$ as $M \to \infty$, \eqref{eq:AE-diff} holds. 

Thus we obtain \eqref{eq:AE-inf} and \eqref{eq:AE-sup}, and the proof of Theorem \ref{thm:AE} is completed. 
We provide numerical evidence for Theorem \ref{thm:AE}  in Subsection \ref{subsec:simulation-AE} below.

 \begin{Rem}
(i) The variance of the maximum likelihood estimator of the parameter $m$ was dealt with by Taylor's unpublished manuscript \cite{Taylor1980}. 
\cite[pp.396--399]{JKB1995} discusses the parameter estimation other than the location. \\ 
(ii) By symmetry, we strongly expect that $\hat{\theta}_n$ is unbiased, that is, $E\left[\hat{\theta}_n \right] = \theta$ for each $n$ such that $\hat{\theta}_n$ is integrable. 
However, to the best of our knowledge, no rigorous proof of this fact has been established. 
We provide numerical evidence for this in Subsection \ref{subsec:unbiased} below. 
 \end{Rem}

\section{Simulation study}\label{sec:SS}

We perform simulation studies using R to illustrate the properties of the maximum likelihood estimator. 
We consider parameter values $m = 0.1 \cdot k$ for $6 \le k \le 15$ and sample sizes $n=10, 50, 100, 500, 1000$. 
In the tables below, the rows correspond to $m$ and the columns to the sample size $n$, unless otherwise stated. 
In the simulations, we let $\theta = 0$. 
We use R version 4.5.2. 
For each  pair $(m, n)$, we generate $N = 10^7$ samples  of size $n$ using the {\tt rpearsonVII()} function in the package `PearsonDS', and  compute the corresponding MLEs.
For the optimization, we use the {\tt nlminb()} function with the sample median as the starting value.

\subsection{Bias}\label{subsec:unbiased}

We compute $\left|E\left[\hat{\theta}_n - \theta \right]\right|$. 
We approximate $E\left[\hat{\theta}_n - \theta \right]$ by $\displaystyle \frac{1}{N} \sum_{i=1}^{N} Z_i$, where $(Z_i)_i$ are i.i.d. random variables distributed as  $\hat{\theta}_n - \theta$.

\begin{table}[H]
\centering
\begin{tabular}{|c|ccccc|}
\hline
 $m \backslash n$ & 10 & 50 & 100 & 500 & 1000 \\
\hline
0.6 & 2.3  & $3.7 \cdot 10^{-5}$  & $8.5 \cdot 10^{-5}$ & $1.1 \cdot 10^{-4}$  & $4.4 \cdot 10^{-5}$   \\
0.7 & $7.8 \cdot 10^{-4}$ & $1.7 \cdot 10^{-5}$ & $1.2 \cdot 10^{-4}$ & $2.6 \cdot 10^{-5}$ & $2.9 \cdot 10^{-6}$  \\
0.8 & $3.3 \cdot 10^{-5}$ & $2.9 \cdot 10^{-5}$ & $1.9 \cdot 10^{-5}$ & $3.2 \cdot 10^{-5}$ & $5.6 \cdot 10^{-6}$  \\
0.9 & $7.3 \cdot 10^{-5}$ & $7.1 \cdot 10^{-5}$ & $5.2 \cdot 10^{-5}$ & $7.2 \cdot 10^{-6}$ & $3.9 \cdot 10^{-5}$  \\
1.0 & $1.8 \cdot 10^{-4}$ & $5.2 \cdot 10^{-5}$ & $6.2 \cdot 10^{-5}$ & $4.5 \cdot 10^{-6}$ & $2.3 \cdot 10^{-5}$ \\
1.1 & $8.6 \cdot 10^{-6}$ & $6.5 \cdot 10^{-5}$ & $2.0 \cdot 10^{-5}$ & $5.2 \cdot 10^{-6}$ & $1.2 \cdot 10^{-5}$ \\
1.2 & $6.0 \cdot 10^{-5}$ & $6.9 \cdot 10^{-5}$ & $4.3 \cdot 10^{-5}$ & $1.4 \cdot 10^{-5}$ & $7.6 \cdot 10^{-6}$   \\
1.3 & $3.7 \cdot 10^{-6}$ & $3.0 \cdot 10^{-6}$ & $8.9 \cdot 10^{-5}$ & $4.5 \cdot 10^{-6}$ & $3.2 \cdot 10^{-6}$   \\
1.4 & $5.7 \cdot 10^{-5}$ & $1.8 \cdot 10^{-5}$ & $2.7 \cdot 10^{-5}$ & $6.3 \cdot 10^{-6}$ & $4.7 \cdot 10^{-6}$  \\
1.5 & $4.2 \cdot 10^{-5}$ & $2.8 \cdot 10^{-5}$ & $6.3 \cdot 10^{-6}$ & $1.6 \cdot 10^{-5}$ & $1.5 \cdot 10^{-6}$  \\
\hline
\end{tabular}
\caption{Simulated values of $\left|E\left[\hat{\theta}_n - \theta\right]\right|$.}\label{tab:UN}
\end{table}

This table suggests that $\hat{\theta}_n$ is {\it not} integrable for $m=0.6$ and $n=10$. 
In this case, it is interesting to compute $P\left(\hat{\theta}_n - \theta > r\right)$ for large $r$. 
We approximate this probability by the proportion of simulated values which are larger than $r$ among the $N = 10^7$ simulated observations. 

\begin{table}[H]
\centering
\begin{tabular}{|c|ccccccc|}
\hline
$r$ & $10^3$ & $5 \cdot 10^3$ & $10^4$ & $5 \cdot 10^4$ & $10^5$ & $5 \cdot 10^5$ & $10^6$ \\
 & $-1.22$ & $-1.08$ & $-1.02$ & $-0.94$ & $-0.94$ & $-0.95$ & $-0.96$  \\
\hline
\end{tabular}
\caption{Simulated values of $\dfrac{\log P(\hat{\theta}_n - \theta > r)}{\log r}$. }\label{tab:tail}
\end{table}

This table suggests that there exists $C > 0$ such that 
$P\left(\hat{\theta}_n - \theta > r\right) \ge Cr^{-1}$ for large $r$, which implies  $\hat{\theta}_n$ is not integrable.

\subsection{Asymptotic normality}\label{subsec:AN-simulation}

Recall that $\varphi_{m}$ is the density of the normal distribution $N\left(0,\dfrac{m + 1}{m (2m -1)} \right)$. 
We consider the Kolmogorov--Smirnov metric: 
\[ \Delta_n \coloneqq \sup_{x \in \mathbb{R}} \left|P\left(\hat{\theta}_n - \theta \le x\right) - \int_{-\infty}^{x} \varphi_m (t) dt \right|. \]
We approximate $P\left(\hat{\theta}_n - \theta \le x\right)$ by the empirical cumulative distribution function $\displaystyle \frac{1}{N} \sum_{i=1}^{N} {\bf 1}_{\{Z_i \le x\}}$ where $(Z_i)_i$ are i.i.d. random variables distributed as  $\hat{\theta}_n - \theta$.  
We recall the Dvoretzky-Kiefer-Wolfowitz-Massart inequality \cite[Theorem 1.32]{Dudley2014}: 
\[ P\left(  \sup_{x \in \mathbb{R}} \left|\frac{1}{N} \sum_{i=1}^{N} {\bf 1}_{\{Z_i \le x\}} - P\left(\hat{\theta}_n - \theta \le x\right)  \right| > \epsilon \right) \le 2\exp(-2N \epsilon^2). \] 

\begin{table}[H]
\centering
\begin{tabular}{|c|ccccc|}
\hline
 $m \backslash n$ & 10 & 50 & 100 & 500 & 1000 \\
\hline
0.6 & $1.1 \cdot 10^{-1}$ & $2.4 \cdot 10^{-2}$ & $1.2 \cdot 10^{-2}$ & $2.4 \cdot 10^{-3}$ & $1.3 \cdot 10^{-3}$ \\
0.7 & $5.3 \cdot 10^{-2}$ & $1.1 \cdot 10^{-2}$ & $5.5 \cdot 10^{-3}$ & $1.2 \cdot 10^{-3}$ & $5.8 \cdot 10^{-4}$ \\
0.8 & $3.4 \cdot 10^{-2}$ & $6.9 \cdot 10^{-3}$ & $3.5 \cdot 10^{-3}$ & $7.9 \cdot 10^{-4}$ & $4.7 \cdot 10^{-4}$ \\
0.9 & $2.5 \cdot 10^{-2}$ & $5.0 \cdot 10^{-3}$ & $2.5 \cdot 10^{-3}$ & $6.4 \cdot 10^{-4}$ & $3.5 \cdot 10^{-4}$ \\
1.0 & $1.9 \cdot 10^{-2}$ & $4.0 \cdot 10^{-3}$ & $2.0 \cdot 10^{-3}$ & $6.4 \cdot 10^{-4}$ & $2.9 \cdot 10^{-4}$ \\
1.1 & $1.5 \cdot 10^{-2}$ & $3.1 \cdot 10^{-3}$ & $1.5 \cdot 10^{-3}$ & $4.2 \cdot 10^{-4}$ & $3.1 \cdot 10^{-4}$ \\
1.2 & $1.3 \cdot 10^{-2}$ & $2.6 \cdot 10^{-3}$ & $1.4 \cdot 10^{-3}$ & $3.0 \cdot 10^{-4}$ & $2.7 \cdot 10^{-4}$ \\
1.3 & $1.1 \cdot 10^{-2}$ & $2.3 \cdot 10^{-3}$ & $1.2 \cdot 10^{-3}$ & $3.4 \cdot 10^{-4}$ & $3.2 \cdot 10^{-4}$ \\
1.4 & $9.3 \cdot 10^{-3}$ & $1.9 \cdot 10^{-3}$ & $9.7 \cdot 10^{-4}$ & $2.8 \cdot 10^{-4}$ & $2.1 \cdot 10^{-4}$  \\
1.5 & $8.2 \cdot 10^{-3}$ & $1.6 \cdot 10^{-3}$ & $9.3 \cdot 10^{-4}$ & $3.8 \cdot 10^{-4}$ & $2.4 \cdot 10^{-4}$  \\
\hline
\end{tabular}
\caption{Simulated values of $\Delta_n$. }\label{tab:AN}
\end{table}

\subsection{Confidence intervals}\label{subsec:CI}

We consider the pivotal quantity $\sqrt{n}\left(\hat{\theta}_n - \theta\right)$ of the model.  
Let $z_{\beta}$ denote the upper $\beta$-quantile, that is, the value satisfying $P\left(\sqrt{n} (\hat{\theta}_n-\theta) \ge z_{\beta}\right) = \beta$ for $0 < \beta < \dfrac{1}{2}$. 
We report the values of $z_{\alpha/2}$ for $\alpha = 0.1, 0.05, 0.01$.
We approximate $z_{\alpha/2}$ by sorting $N = 10^7$ MLEs and using the {\tt quantile} function in R. 
Using $z_{\alpha/2}$, we obtain the $100(1-\alpha)\%$ confidence interval for the location parameter $\theta$: 
\[ \left[\hat{\theta}_n - \frac{z_{\alpha/2}}{\sqrt{n}}, \hat{\theta}_n + \frac{z_{\alpha/2}}{\sqrt{n}}\right].\]

In the following tables, the column labeled $\infty$ reports the upper $\alpha/2$ quantile of the limiting normal distribution $N\left(0,\dfrac{m+1}{m(2m-1)} \right)$ given by Theorem~\ref{thm:CLT}. 

\begin{table}[H]
\centering
\begin{tabular}{|c|ccccc|c|}
\hline
 $m \backslash n$ & 10 & 50 & 100 & 500 & 1000 & $\infty$ \\
\hline
0.6 & 19.60 & 7.06 & 6.46 & 6.08 & 6.05 & 6.01 \\
0.7 &  6.34 & 4.34 & 4.19 & 4.08 & 4.06 & 4.05 \\
0.8 &  4.09 & 3.33 & 3.25 & 3.20 & 3.19 & 3.19 \\
0.9 &  3.17 & 2.75 & 2.71 & 2.68 & 2.68 & 2.67 \\
1.0 &  2.64 & 2.38 & 2.35 & 2.33 & 2.33 & 2.33 \\
1.1 &  2.29 & 2.11 & 2.10 & 2.08 & 2.08 & 2.07 \\
1.2 &  2.04 & 1.91 & 1.90 & 1.89 & 1.88 & 1.88 \\
1.3 &  1.85 & 1.75 & 1.74 & 1.73 & 1.73 & 1.73 \\
1.4 &  1.70 & 1.62 & 1.61 & 1.61 & 1.61 & 1.61 \\
1.5 &  1.58 & 1.52 & 1.51 & 1.50 & 1.50 & 1.50 \\
\hline
\end{tabular}
\caption{Simulated values of $z_{\alpha/2}$ ($\alpha = 0.1$). }\label{tab:CI1}
\end{table}

\begin{table}[H]
\centering
\begin{tabular}{|c|ccccc|c|}
\hline
 $m \backslash n$ & 10 & 50 & 100 & 500 & 1000 & $\infty$ \\
\hline
0.6 & 37.96 & 9.14 & 7.96 & 7.29 & 7.22 & 7.16 \\
0.7 &  9.46 & 5.34 & 5.07 & 4.87 & 4.85 & 4.83 \\
0.8 &  5.57 & 4.04 & 3.91 & 3.82 & 3.80 & 3.80 \\
0.9 &  4.12 & 3.33 & 3.25 & 3.20 & 3.19 & 3.18 \\
1.0 &  3.35 & 2.87 & 2.82 & 2.78 & 2.78 & 2.77 \\
1.1 &  2.87 & 2.54 & 2.51 & 2.48 & 2.48 & 2.47 \\
1.2 &  2.53 & 2.29 & 2.27 & 2.25 & 2.25 & 2.24 \\
1.3 &  2.28 & 2.10 & 2.08 & 2.06 & 2.06 & 2.06 \\
1.4 &  2.09 & 1.94 & 1.93 & 1.92 & 1.92 & 1.91 \\
1.5 &  1.93 & 1.82 & 1.80 & 1.79 & 1.79 & 1.79 \\
\hline
\end{tabular}
\caption{Simulated values of $z_{\alpha/2}$ ($\alpha = 0.05$). }\label{tab:CI2}
\end{table}

\begin{table}[H]
\centering
\begin{tabular}{|c|ccccc|c|}
\hline
 $m \backslash n$ & 10 & 50 & 100 & 500 & 1000 & $\infty$ \\
\hline
0.6 & 162.34 & 15.48 & 11.47 &  9.73 &  9.57 & 9.41 \\
0.7 &  21.55 &  7.66 &  6.91 &  6.45 &  6.39 & 6.35 \\
0.8 &  10.48 &  5.58 &  5.27 &  5.04 &  5.01 & 4.99 \\
0.9 &   6.99 &  4.53 &  4.35 &  4.21 &  4.20 & 4.18 \\
1.0 &   5.31 &  3.87 &  3.75 &  3.67 &  3.65 & 3.64 \\
1.1 &   4.35 &  3.41 &  3.33 &  3.27 &  3.26 & 3.25 \\
1.2 &   3.74 &  3.07 &  3.01 &  2.96 &  2.95 & 2.95 \\
1.3 &   3.29 &  2.80 &  2.75 &  2.72 &  2.71 & 2.71 \\
1.4 &   2.97 &  2.59 &  2.55 &  2.52 &  2.52 & 2.51 \\
1.5 &   2.71 &  2.41 &  2.38 &  2.36 &  2.35 & 2.35 \\
\hline  
\end{tabular}
\caption{Simulated values of $z_{\alpha/2}$ ($\alpha = 0.01$).}\label{tab:CI3}
\end{table}

In the case  $m=1$, \cite[Section 4]{Freue2007} provides a table of quantiles $z_{\alpha/2}$ for sample sizes $n = 5, 10, \dots, 40$ and $\alpha = 0.1, 0.05, 0.01$.

\subsection{Asymptotic efficiency}\label{subsec:simulation-AE}

We consider the quantity $n E\left[\left(\hat{\theta}_n - \theta\right)^2\right]$ appearing in Theorem \ref{thm:AE}. 
The column labeled $\infty$ reports the theoretical limit $\dfrac{m+1}{m(2m-1)}$ given by Theorem~\ref{thm:AE}. 

\begin{table}[H]
\centering
\begin{tabular}{|c|ccccc|c|}
\hline
 $m \backslash n$ & 10 & 50 & 100 & 500 & 1000 & $\infty$\\
\hline
    0.6 & NA     & 25.756 & 16.108 & 13.756 & 13.537 &13.333 \\
    0.7 & NA     & 7.231 & 6.561 & 6.162 & 6.114 & 6.071 \\
    0.8 & 8.979 & 4.154 & 3.933 & 3.786 & 3.769 & 3.750 \\
    0.9 & 4.511 & 2.831 & 2.730 & 2.657 & 2.648 & 2.639 \\
    1.0 & 2.907 & 2.110 & 2.053 & 2.011 & 2.005 & 2.000\\
    1.1 & 2.106 & 1.660 & 1.623 & 1.598 & 1.595 & 1.591\\
    1.2 & 1.631 & 1.356 & 1.332 & 1.314 & 1.311 & 1.310\\
    1.3 & 1.322 & 1.138 & 1.122 & 1.108 & 1.107 & 1.106\\
    1.4 & 1.105 & 0.976 & 0.964 & 0.955 & 0.953 & 0.952 \\
    1.5 & 0.946 & 0.851 & 0.842 & 0.835 & 0.834 & 0.833 \\
\hline
\end{tabular}
\caption{Simulated values of $n\,\mathbb{E}\!\left[\left(\hat\theta_n-\theta\right)^2\right]$. 
}\label{tab:AE}
\end{table}

In the case  $m=1$, \cite[Table 2]{Freue2007} provides numerical values for $n = 5,6, \dots, 14, 15, 20, 25, \dots, 35, 40$. 
These results are consistent with the numerical results in \cite[Table 2]{AOO2022AISM} for the joint estimation of the location and scale.  \\

{\it Acknowledgements.} \ \ The author would like to express his gratitude to the referees for their helpful comments and suggestions.\\

\bibliographystyle{plain}
\bibliography{locMLE-Cauchy}

\end{document}